\documentclass[11pt,reqno]{amsart}

\usepackage[utf8]{inputenc}
\usepackage{setspace}
\usepackage{geometry}
\usepackage{enumerate}
\usepackage{enumitem, xcolor, amssymb,latexsym,amsmath,bbm}
\usepackage{mathtools}
\usepackage[mathscr]{euscript}
\usepackage{amsmath}
\usepackage{amssymb}
\usepackage{enumitem}

\usepackage{tikz}
\usepackage{wrapfig}
\usepackage{enumerate}
\usepackage{graphicx}
\usepackage{subfigure}

\usetikzlibrary{arrows.meta}

\usetikzlibrary{decorations.markings}

\usepackage[colorlinks=true,citecolor=blue,
linkcolor=blue,urlcolor=blue]{hyperref}

\newtheorem{theorem}{Theorem}
\newtheorem{theorem*}{Theorem}
\newtheorem{lemma}{Lemma}
\newtheorem{proposition}[theorem]{Proposition}
\newtheorem{definition}{Definition}
\newtheorem{corollary}{Corollary}
\newtheorem{remark}{Remark}

\newcommand{\RR}{\mathbb{R}}

\numberwithin{equation}{section}
\numberwithin{theorem}{section}
\numberwithin{lemma}{section}
\numberwithin{example}{section}
\allowdisplaybreaks

\makeatletter
\@namedef{subjclassname@2020}{\textup{2020} Mathematics Subject Classification}
\makeatother

\begin{document}
    
    \title[Nonlinear elliptic eigenvalue problems]{Nonlinear elliptic eigenvalue problems in cylindrical domains becoming unbounded in one direction}

    \author{Rama Rawat}
    \email{rrawat@iitk.ac.in}

    \author{Haripada Roy}
    \email{haripada@iitk.ac.in}

    \author{Prosenjit Roy}
    \email{prosenjit@iitk.ac.in}
    
    \address{Indian Institute of Technology Kanpur, India}

    .
    \date{}

    \smallskip
    \begin{abstract}
    The aim of this work is to characterize the asymptotic behaviour of the first eigenfunction  of the generalised $p$-Laplace operator with mixed (Dirichlet and Neumann) boundary conditions in cylindrical domains when the length of the cylindrical domains tends to infinity. This generalises an earlier work of Chipot et.al. \textit{[Asymptotics of eigenstates of elliptic problems
with mixed boundary data on domains tending to infinity. Asymptot. Anal., 85(3-4):199–227,
2013]} where the linear case $p=2$ is studied.  Asymptotic behavior of all the higher eigenvalues of the linear case and the second eigenvalues of general case (using topological degree) for such problems is also studied.  
    \end{abstract}
    
    \maketitle
    

\section{Introduction}
The asymptotic analysis of problems set in cylindrical domains which become unbounded in one or several directions for various partial differential equations has been carried out by several authors since the pioneering work of Chipot \cite{ChPrIt}. In this paper, we revisit one of these problems which is set in cylindrical domains becoming unbounded in one direction for the generalised p-Laplacian with mixed boundary conditions. This continues the work in \cite{ChPrIt}, which concerns the asymptotic analysis of the linear case $p=2$ with mixed (Dirichlet-Neumann) boundary conditions. We now briefly describe the setting.

Consider $\Omega_\ell=(-\ell,\ell)\times\omega$, where $\ell$ is a positive real number and $\omega$ is an open bounded set in $\mathbb{R}^{n-1}$, $n\geq2$. Any generic point $x\in \mathbb{R}^n$ will be denoted by $x=(x_1,X_2)$, where $x_1\in \mathbb{R}$ and $X_2\in\mathbb{R}^{n-1}$. Then for each $x_1\in (-\ell,\ell)$, $\omega$ is the cross-section of $\Omega_\ell$ orthogonal to $x_1$-direction. Let $A$ be a $n\times n$ symmetric matrix of the form
$$A(X_2):=\begin{pmatrix}
a_{11}(X_2) & A_{12}(X_2)\\
A_{12}^t(X_2) & A_{22}(X_2)
\end{pmatrix},  \ X_2\in \omega,$$
where $A_{22}$ is an $(n-1)\times(n-1)$ matrix and assume that $A$ is uniformly bounded and uniformly elliptic on $\mathbb{R}\times\omega$ ( precise definitions are given in next section). For $p\geq2$ consider the following eigenvalue problem on $\Omega_\ell$ for the generalised p-Laplacian with Dirichlet boundary condition
\begin{equation}\label{dir}
    \begin{cases}
    -\textrm{div}\left(|A(X_2)\nabla u_\ell\cdot \nabla u_\ell|^{\frac{p-2}{2}}A(X_2)\nabla u_\ell\right)=\lambda_{D}(\Omega_\ell)|u_\ell|^{p-2}u_\ell \ \textrm{in} \ \Omega_{\ell},\\
    u_\ell=0 \ \textrm{on} \ \partial\Omega_{\ell},
    \end{cases}
\end{equation}
and the corresponding eigenvalue problem
\begin{equation}\label{cross}
\begin{cases}
 -\textrm{div}\left(|A_{22}(X_2)\nabla_{X_2} u\cdot \nabla_{X_2} u|^{\frac{p-2}{2}}A_{22}(X_2)\nabla_{X_2} u\right)=\mu(\omega)|u|^{p-2}u\ \textrm{in} \ \omega,\\
    u=0\ \textrm{on} \  \partial\omega,
    \end{cases}
\end{equation}
defined on the cross section $\omega$, where $\nabla_{X_2}=(\frac{\partial}{\partial x_2},\frac{\partial}{\partial x_3},\cdots,\frac{\partial}{\partial x_n})$.\smallskip

For the linear case $p=2,$ it is well known that each of the problems \eqref{dir} and \eqref{cross} admit an infinite set of positive discrete eigenvalues tending to infinity (see \cite{Kesavan}, Theorem 3.6.1). Let $\lambda_{D}^k(\Omega_\ell)$ and $\mu^k(\omega)$ denote the $k$-th eigenvalues of \eqref{dir} and \eqref{cross} respectively in this case. In \cite{ChRo}, Theorem 2.4, it was shown that as $\ell$ goes to infinity, the $k$-th eigenvalue $\lambda_{D}^k(\Omega_\ell)$  of \eqref{dir}  converges to the first eigenvalue $\mu^1(\omega)$ of \eqref{cross} with the optimal convergence rate of $\frac{1}{\ell^2}$. It was further proved that under appropriate scaling, the eigenfunctions corresponding to the first eigenvalues of \eqref{dir} converge to the eigenfunction corresponding to the first eigenvalue of \eqref{cross} in the appropriate function space (see \cite{ChRo}, Theorem 3.4). In general for $p\geq2$, the first eigenvalue of \eqref{dir} and \eqref{cross} can be obtained by appropriate minimization problems (described later in \eqref{2.6} and \eqref{2.7}), we continue to denote them by $\lambda_{D}^1(\Omega_\ell)$ and $\mu^1(\omega)$ respectively, the value of $p$ will be clear from the context. In \cite{EsLuPrFi}, the following theorem  was proved. 

\begin{theorem*}\emph{\cite{EsLuPrFi}}
 Let $p\geq2$. Then $\exists$ a constant $C$ depends only on $A$, $\omega$ and $p$, such that
 \begin{equation}\label{result firoj 1}
 \mu^1(\omega)\leq \lambda_{D}^1(\Omega_\ell)\leq\mu^1(\omega)+\frac{C}{\ell}.
    \end{equation}
\end{theorem*}

 We now describe the problem which will be our main focus in this paper.  Consider the following eigenvalue problem for the  generalised p-Laplacian with  mixed (Dirichlet and Neumann) boundary conditions:

\begin{equation}\label{mix}
    \begin{cases}
    -\textrm{div}\left(|A(X_2)\nabla u_\ell\cdot \nabla u_\ell|^{\frac{p-2}{2}}A(X_2)\nabla u_\ell\right)=\lambda_{M}(\Omega_\ell)|u_\ell|^{p-2}u_\ell\ \textrm{in} \ \Omega_{\ell},\\
    u_\ell=0 \ \textrm{on} \  \gamma_\ell:=(-\ell,\ell)\times\partial\omega,\\
    \big(A(X_2)\nabla u_\ell\big)\cdot \nu=0 \ \textrm{on} \ \Gamma_\ell:=\{-\ell,\ell\}\times\omega,
    \end{cases}
\end{equation}
where $\nu$ denotes the outward unit normal to $\Gamma_\ell$. 


In the linear case $p=2$, where the set of all eigenvalues of \eqref{mix} is well understood (see \cite{LoVlRo}) a similar asymptotic analysis of the first and second eigenvalue was carried out in \cite{ChPrIt}. It was shown there that under an additional condition, the first eigenvalue $\lambda_{M}^1(\Omega_\ell)$ of \eqref{mix} converges to a limit as $\ell$ goes to infinity but in this case the limit is strictly smaller than the first eigenvalue $\mu^1(\omega)$ of  \eqref{cross}. In particular there is a gap in the limiting behaviour of $\lambda_{M}^1(\Omega_\ell)$ and $\mu^1(\omega).$ This is in  sharp contrast to the limiting behavior of  eigenvalue problem with Dirichlet boundary condition \eqref{dir} as discussed before. Further it was proved that under a symmetry condition on the matrix $A$, the second eigenvalue $\lambda_{M}^2(\Omega_\ell)$ of \eqref{mix} has the same limit as that of the first eigenvalue $\lambda_{M}^1(\Omega_\ell)$ of \eqref{mix} as $\ell$ goes to infinity.

This was followed up in \cite{EsLuPrFi} where it was shown that  the gap phenomenon continues to hold for the limit superior of the first eigenvalues of \eqref{mix}  for the nonlinear case  $p>2$ under a similar additional condition. To summarize these results:
\smallskip 

\begin{theorem*}\emph{\cite{ChPrIt,EsLuPrFi}}\label{theorem itai and firoj}
 Let $p\geq2$ and let $W$  be the eigenfunction corresponding to the eigenvalue $\mu^1(\omega)$ such that $\int_{\omega}
 |W |^p = 1.$ Then\\
\emph{(i)} For $p=2$,
\begin{equation}\label{result Itai}
\lim_{\ell\rightarrow\infty}\lambda_{M}^1(\Omega_\ell)<\mu^1(\omega),
\end{equation}
provided $A_{12}\cdot \nabla_{X_2}W\not\equiv0$ a.e. on $\omega$, otherwise $\lambda_{M}^1(\Omega_\ell)=\mu^1(\omega)$ for all $\ell>0$. Furthermore, if $A$ satisfies the symmetry (S) (see Definition \ref{definition S} later), then $$\lim_{\ell\rightarrow\infty}\lambda_{M}^2(\Omega_\ell)=\lim_{\ell\rightarrow\infty}\lambda_{M}^1(\Omega_\ell).$$
\emph{(ii)} For $p>2$, one has
\begin{equation}\label{result firoj 2}
\limsup_{\ell\rightarrow\infty}\lambda_{M}^1(\Omega_\ell)<\mu^1(\omega)
\end{equation}
provided $A_{12}\cdot \nabla_{X_2}W\not\equiv0$ a.e. on $\omega$, otherwise $\lambda_{M}^1(\Omega_\ell)=\mu^1(\omega)$ for all $\ell>0$.
\end{theorem*}
A detailed discussion on the condition $A_{12}\cdot \nabla_{X_2}W\not\equiv 0$ is done in \cite{ChPrIt} [see, Remark 4.2].

\noindent On the way to the proof of the gap phenomenon in Theorem \ref{theorem itai and firoj} (i), several other interesting results were obtained in \cite{ChPrIt}. For example, behavior of the first normalised eigenfunction of \eqref{mix} near the end (on sets like $\Omega_\ell\setminus\Omega_{\ell-1}$), in the middle of the cylinder (sets like $\Omega_{\ell/2}$) were established. Also the value $\Lambda := \limsup_{\ell\rightarrow\infty}\lambda_M^1(\Omega_\ell)$ is identified with an appropriate problem on the semi-infinite cylinders. \smallskip

Similar questions remained unanswered for the case $p>2$, which we wish to take up in this paper.
Let $u_\ell$ be the normalised eigenfunction corresponding to the first eigenvalue $\lambda_M^1(\Omega_\ell)$ of \eqref{mix}. 
In our first result, we analyse the behavior of $u_\ell$ as $\ell$ goes to infinity and show that as in the case $p=2$ these eigenfunctions decay to zero in the middle of the cylinder and their mass concentrates near the base of the cylinder.
Throughout this paper, $[x]$ will denote the integer part of $x \in \mathbb{R}.$

\begin{theorem}\label{theorem decay}
Assume that $p\geq2$ and $A_{12}\cdot \nabla_{X_2}W\not\equiv0$ a.e in $\omega$. Then $\exists$ constants $\alpha\in (0,1)$ and $C_1,C_2\geq0$ depend only on $A$, $\omega$ and $p$ such that for $\ell>\ell_0$ we have for every $0<r\leq \ell-1$,
\begin{equation}\label{decay of grad}
    \int_{\Omega_r}|\nabla u_\ell|^p\leq C_1\alpha^{[\ell-r]},
\end{equation}
and
\begin{equation}\label{decay of u}
    \int_{\Omega_r}| u_\ell|^p\leq C_2\alpha^{[\ell-r]}.
\end{equation}
\end{theorem}

\smallskip

 The following identity 
$$A\nabla(u v)\cdot\nabla(u v)-u^2\left(A\nabla v\cdot\nabla v\right)= A\nabla u \cdot \nabla (uv^2) $$
plays a crucial role in proofs for the case $p=2$. For general $p$, we do not have this kind of identity. To overcome this difficulty we  use the uniformly ellipticity condition on the matrix $A$ and the Poincar\'e inequality (see \eqref{uniform ellip} and \eqref{poincare} in Section 2). Other complication arises because of the fact, that for  $p=2$ the spaces involved are Hilbert spaces which is not the case for general $p$. Our next result is crucial for understanding the value $\Lambda := \limsup_{\ell\rightarrow\infty}\lambda_M^1(\Omega_\ell)$. 

 Consider the semi-infinite cylinders 
$$\Omega_{\infty}^+=(0,\infty)\times \omega\ \mathrm{and} \ \Omega_{\infty}^-=(-\infty,0)\times \omega$$ with the boundaries
$$\gamma_{\infty}^+=(0,\infty)\times \partial\omega,\ \ \gamma_{\infty}^-=(-\infty,0)\times\partial \omega,$$
and the quantities
\begin{equation}\label{nu infinity}
    \nu_{\infty}^{\pm}=\inf_{0\neq u\in V(\Omega_{\infty}^{\pm})}\frac{\int_{\Omega_{\infty}^{\pm}}|A\nabla u\cdot \nabla u|^{\frac{p}{2}}}{\int_{\Omega_{\infty}^{\pm}}|u|^p},
\end{equation}
where the spaces $V(\Omega_{\infty}^{\pm})$ are defined by
$$V(\Omega_{\infty}^{\pm})=\{u\in W^{1,p}(\Omega_{\infty}^{\pm}): u=0\ \textrm{on} \ \gamma_{\infty}^{\pm}\}.$$
Detailed description of these spaces is given in Section 2. 



\begin{theorem}\label{theorem limit of first eigenvalue}
Assume that $p\geq2.$ Then
\begin{equation}\label{limit of first eigenvalue}
    \lim_{\ell \rightarrow\infty}\lambda_M^1(\Omega_\ell )=\min \left\{ \nu_{\infty}^+,\nu_{\infty}^-\right\}.
\end{equation}
\end{theorem}
We further provide a condition for which the gaps hold between $\nu_{\infty}^{\pm}$ and $\mu^1(\omega)$ and another condition for which the gaps fail. This was known for $p=2$ from \cite{ChPrIt}.

\begin{theorem}\label{theorem gap holds for mu}
\emph{(i)} Assume that $A_{12}\cdot\nabla_{X_2}W\not\equiv0$. If
\begin{equation}\label{gap holds for mu}
   \int_{\omega}|A_{22}\nabla_{X_2} W\cdot\nabla_{X_2} W|^{\frac{p-2}{2}}(A_{12}\cdot\nabla_{X_2} W)W\geq0,
\end{equation}
then $\nu_{\infty}^{+}< \mu_1(\omega)$.

\emph{(ii)} If
\begin{equation}\label{no gap for mu}
    A_{12}\cdot\nabla_{X_2} W\leq0\ \mathrm{a.e}\ \mathrm{in} \ \omega,
\end{equation}
then $\nu_{\infty}^{+}= \mu_1(\omega)$. Moreover in this case there is no minimizer realizing $\nu_{\infty}^+$.
\end{theorem}

\smallskip

We conclude the paper with some results about the convergence of the other eigenvalues of problem \eqref{mix}, these results may be compared with the ones in \cite{ChPrIt}, \cite{ChRo}. Unlike the linear case $p=2,$
 there is no specific description of the second and higher eigenvalues available for $p>2$ for the eigenvalue problem \eqref{mix}.
Although there is a description available of what we call Krasnoselskii eigenvalues, through the concept of Krasnoselskii's genus. Let $\{\beta_k(\Omega_\ell)\}$ denote these Krasnoselskii eigenvalues of \eqref{mix}
which we instead use to analyze the asymptotic behaviour. We discuss these eigenvalues briefly in the next section. For $p=2$ therefore, we have two descriptions of eigenvalues of the problem \eqref{mix}, one through classical methods using Hilbert space structure and the other one using Krasnoselskii's genus, and these two descriptions coincide in this case (see Lemma \ref{lemma 6.1}). For $p>2$ it is not clear if Krasnoselskii eigenvalues exhaust the set of all eigenvalues of \eqref{mix}. However the first Krasnoselskii eigenvalue $\beta_1(\Omega_\ell)$ coincides with the usual first eigenvalue $\lambda_M^1(\Omega_\ell)$ (see Corollary \ref{coro 2}).\smallskip

In the literature there are other descriptions of eigenvalues of \eqref{mix} available, interested readers may look at \cite{drabek} for example.


\begin{definition}\label{definition S}
We shall say that $A$ satisfies the symmetry (S) if the set $\omega$ is symmetric w.r.t the origin, i.e. $\omega=-\omega$ and $A(-X_2)=A(X_2)$.
\end{definition}
The following theorem is an extension of Theorem 7.1 in \cite{ChPrIt} for higher eigenvalues. Only the asymptotic behavior of second eigenvalue was considered there. 
\begin{theorem}\label{theorem kth ev}
Let $A$ satisfies the symmetry (S). Then for $p=2$ and for any $k\geq2$
    \begin{equation}\label{kth ev}
\lim_{\ell\rightarrow\infty}\lambda_M^k(\Omega_\ell)=\lim_{\ell\rightarrow\infty}\lambda_M^1(\Omega_\ell).
    \end{equation}
\end{theorem}

\begin{theorem}\label{theorem second K ev}
 Let $A$ satisfies the symmetry (S). Then for  $p>2,$
\begin{equation}\label{second K ev}
\lim_{\ell\rightarrow\infty}\beta_2(\Omega_\ell)=\lim_{\ell\rightarrow\infty}\lambda_M^1(\Omega_\ell).
\end{equation}
   \end{theorem}

For other related work on asymptotic behavior of problems defined on cylindrical domains interested readers may look at \cite{ChRo1,Chipot2,ChSo,ChYie,ChZu} for study of elliptic problems, \cite{ChRo3,ChRo2,ChXie} for parabolic problems. The behavior of Stokes equations has been studied in \cite{ChSo3,ChSo2},  whereas purely variational problems have been  studied in \cite{Chipot3,ChMo,Dret}. Recently, problems involving nonlocal operators have also been considered, e.g., see \cite{ChAl,Gy,InPr,IcPr2,YeKa}. For the first time in such literature semilinear problems are studied in  \cite{davilla}.

This paper is organized as follow:  In Section 2 we  briefly review the notion of Krasnoselskii's genus and related facts, underlying function spaces, uniform Poincar\'e inequality, properties of the matrix $A$ and other relevant known results. Section 3 contains the proof of Theorem \ref{theorem decay}. In Section 4 we prove Theorem \ref{theorem limit of first eigenvalue}, which identifies the limit of $\lambda_M^1(\Omega_\ell)$ as $\ell$ goes to infinity. In Section 5 gap phenomenon on semi infinite cylinder (Theorem \ref{theorem gap holds for mu}) is studied.  In Section 6 we conclude the paper with the convergence results (Theorem \ref{theorem kth ev} and \ref{theorem second K ev}) concerning the $k$-th eigenvalue for $p=2$ and the second Krasnoselskii's eigenvalue for $p>2$ of the problem \eqref{mix}.\smallskip

\section{Preliminaries}


Let $V$ be a Banach space. Define the collection
$$\mathcal{A}=\left\{A\subset V: A\ \mathrm{is} \ \mathrm{closed} \ \mathrm{and} \ A=-A\right\}$$
of all closed and symmetric subsets of $V$. Let $\gamma:\mathcal{A}\rightarrow \mathbb{N}\cup\{0,\infty\}$ be defined by
$$\gamma(A)=\begin{cases}\mathrm{inf}\left\{m:\exists \, h:A\rightarrow \mathbb{R}^m\setminus \{0\},h\ \mathrm{is} \ \mathrm{continuous} \ \mathrm{and} \ h(-u)=-h(u) \right\}\\
\infty,\ \mathrm{if} \ \{\cdots \}=\phi,\ \mathrm{in} \ \mathrm{particular \, if} \ A\ni 0,
\end{cases}$$
for $A\neq \phi$ and $\gamma(\phi)=0$. $\gamma(A)$ is called the Krasnoselskii's genus of the set $A$. We summarise some important properties of the Krasnoselskii's genus in the following :
\begin{proposition}\label{prop genus}
\emph{(i)} $\gamma$ is a monotone sub-additive continuous map and $\gamma(A)$ is finite for any compact symmetric subset of $V\setminus \{0\}$.\\
\emph{(ii)} Let $\Omega$ be any bounded symmetric neighborhood of the origin in $\mathbb{R}^m$. Then $\gamma(\partial\Omega)=m$.\\
 \emph{(iii)} Let $A$ be a compact symmetric subset of a Hilbert space $V$ with the inner product $(\cdot,\cdot)_V$ and let $\gamma(A)=m<\infty$. Then $A$ contains at least $m$ mutually orthogonal vectors, i.e. $\exists u_1, u_2,\cdots,u_m\in A$ such that $(u_j,u_k)_V=0$ for $j\neq k$. 
 \end{proposition}\smallskip

Suppose $E$ is an even functional on $V$ i.e., $E(-u)=E(u)$, $u\in V$. A sequence $(u_m)$ is called a Palais-Smale sequence for $E$, if $|E(u_m)|\leq c$ uniformly on $m$, while $\lVert DE(u_m)\rVert\rightarrow0$ as $m\rightarrow \infty$, where $DE(u)$ is the Fr\'echet derivative of the functional $E$ at $u$. $E$ is said to satisfy the Palais-Smale (P.S) condition if any Palais-Smale sequence has a strongly convergent subsequence. The following theorem assures the existence of the critical points and characterizes the critical values of $E$.
\begin{theorem}\label{theorem genus}
\emph{(i)} Suppose $E$ is an even functional on a complete symmetric $C^{1,1}$-manifold $M\subset V\setminus \{0\}$ of some Banach space $V$. Also suppose $E$ satisfies (P.S) condition and is bounded from below on $M$.
Let $\hat{\gamma}(M)=\mathrm{sup}\left\{\gamma(K):K\subset M\ \mathrm{compact} \ \mathrm{and}\ \mathrm{symmetric} \right\}.$ Then the functional $E$ possesses at least $\hat{\gamma}(M)\leq\infty$ pairs of critical points.\\
\emph{(ii)} Let $\mathcal{A}$ be defined as above. For any $k\leq \hat{\gamma}(M)$ consider the family 
$$\mathcal{F}_k=\left\{A\in\mathcal{A}:A\subset M, \gamma(A)\geq k \right\}.$$
 If $$\beta_k=\underset{A\in\mathcal{F}_k}{\mathrm{inf}}\,\underset{u\in A}{\mathrm{sup}} \,E(u)$$ is finite, then $\beta_k$ is a critical value of $E$.
\end{theorem}
We refer to \cite{Struwe} (Chapter II, Section 5) for more details about the Krasnoselskii's genus and the proofs of Proposition \ref{prop genus} and Theorem \ref{theorem genus}. \smallskip

For an open set $\Omega$ in $\mathbb{R}^n$ and $1<p<\infty$, $W^{1,p}(\Omega)$ is the usual Sobolev space defined by
$$W^{1,p}(\Omega)=\{v\in L^p(\Omega):\partial_{x_i}v\in L^p(\Omega),i=1,2\cdots,n\},$$
with the norm
$$\lVert v \rVert_{W^{1,p}(\Omega)}=\Big\{\int_{\Omega}(|v|^p+|\nabla v|^p)\Big\}^{1/p}.$$
As before, let $\Omega_\ell=(-\ell,\ell)\times\omega$ on $\mathbb{R}^n$, where $\ell>0$ and $\omega$ is an open bounded set in $\mathbb{R}^{n-1}$ and we define the space 
 $$W_0^{1,p}(\Omega_\ell)=\{v\in W^{1,p}(\Omega_\ell):v=0 \ \textrm{on} \ \partial\Omega_\ell\}.$$
We denote the boundary of $\Omega_\ell$ by $\gamma_\ell\cup\Gamma_\ell$, where $\gamma_\ell=(-\ell,\ell)\times\partial\omega$ and $\Gamma_\ell=\{-\ell,\ell\}\times \omega$.
To analyze the eigenvalues and eigenfunctions of the problem \eqref{mix}, we define a suitable space
$$V(\Omega_\ell)=\{v\in W^{1,p}(\Omega_\ell):v=0 \ \textrm{on} \ \gamma_\ell\}.$$
The boundary conditions are considered in the sense of traces.
The following version of Poincar\'e inequality is used in this work.
\begin{lemma}
 \label{theorem poincare}\emph{\textbf{(Poincar\'e inequality)}}
Let $\Omega=U\times \omega$, where $U$ is any open subset of $\mathbb{R}$, $\omega$ is an open bounded set in $\mathbb{R}^{n-1}$. Let $u\in V(\Omega)$, where $V(\Omega)=\{u\in W^{1,p}(\Omega): u=0$ on $U\times \partial \omega\}$. Then we have
\begin{equation}\label{poincare}
    \lVert u \rVert_{L^p(\Omega)}\leq C_p\lVert \nabla u \rVert_{L^p(\Omega)}
\end{equation}
where $C_p>0$ is a constant depends only on $p$ and $\omega$.
\end{lemma}
Using Lemma \ref{theorem poincare} it can be easily verified that $\lVert v \rVert^p_{p,\Omega_\ell}=\int_{\Omega_\ell}|\nabla v|^p$ is a norm on $V(\Omega_\ell)$ and $V(\Omega_\ell)$ is a Banach space with respect to this norm.\smallskip

We consider the $n\times n$ symmetric matrix $A$ of the form
$$A(X_2)=\begin{pmatrix}
a_{11}(X_2) & A_{12}(X_2)\\
A_{12}^t(X_2) & A_{22}(X_2)
\end{pmatrix}$$
where $ a_{11}(X_2) \in \mathbb R,$ \, $A_{12}(X_2)$ is a 
$1 \times (n-1)$ matrix and $A_{22}(X_2)$ is a
$(n-1) \times (n-1)$ matrix. The components of $A(X_2)$ are assumed to be bounded
measurable functions on $\omega$ and we assume that   $A$ is an uniformly bounded and uniformly positive definite matrix i.e., there exists positve constants $C$ and $\lambda$ such that
\begin{equation}\label{uniform bound}
    \lVert A(X_2) \rVert\leq C \ \textrm{a.e} \ X_2\in \omega,
\end{equation}
and
\begin{equation}\label{uniform ellip}
    A(X_2)\xi\cdot\xi \geq\lambda |\xi|^2 \ \textrm{a.e} \ X_2\in \omega,\ \ \forall\xi\in\mathbb{R}^n.
\end{equation}
Here $\lVert \cdot \rVert$ denotes the norm of matrices, $|.|$ denotes the euclidean norm and $ X \cdot Y$ denotes the usual inner product 
 of two vectors $X$ and $Y$ in $\mathbb{R}^n$.\smallskip

To obtain the Krasnoselskii eigenvalues  of \eqref{mix}, we define a closed symmetric $C^{1,1}$-manifold $M_\ell$ of the Banach space $V(\Omega_\ell)$ as
$$M_\ell=\left\{u\in V(\Omega_\ell):\int_{\Omega_\ell}|u|^p=1 \right\},$$ 
and an even functional $E_\ell$ on $M_\ell$ by
$$E_\ell(u)=\int_{\Omega_\ell}|A\nabla u\cdot \nabla u|^{\frac{p}{2}}.$$
Then $E_\ell$ satisfies (P.S) condition, as shown in \cite{drabek}. For $k\in\mathbb{N}$, let $\{v_1,v_2,\cdots,v_k\}$ be a linearly independent set in $M_\ell$ and let $K=\left\{v=\sum_{i=1}^{k}\alpha_iv_i:\alpha_i\in\mathbb{R},1\leq i\leq k\ \mathrm{and}\ \int_{\Omega_\ell}|v|^p=1\right\}$. Then  $K$ is a compact and symmetric subset of a $k-$dimensional subspace of $L^p$ and is homeomorphic to $S^{k-1},$ the unit sphere in $\mathbb{R}^k.$ Therefore by Proposition \ref{prop genus} (ii), $\gamma(K)=k$ and hence $\hat{\gamma}(M_\ell)=\infty.$  Theorem \ref{theorem genus} (i) assures that there are infinitely many critical values in both the cases $p=2$ and $p>2$. For $p=2$, these critical values of \eqref{mix} coincide with the usual eigenvalues (as discussed earlier) and therefore we continue to  denote them by $\lambda_M^k(\Omega_\ell)$ and for $p>2$, we call these critical values as Krasnoselskii eigenvalues of the problem \eqref{mix} and denote them by $\beta_k(\Omega_\ell)$.\smallskip


 




Multiplying the equations \eqref{dir} and \eqref{cross} by functions from the corresponding underlying spaces and then integrating by parts, we formulate the corresponding weak version of the Dirichlet problems \eqref{dir} and \eqref{cross} respectively as follows.
\begin{equation}\label{dir-weak}
    \begin{cases}u\in W_0^{1,p}(\Omega_\ell),\\
    \int_{\Omega_\ell}|A\nabla u\cdot \nabla u|^{\frac{p-2}{2}}A\nabla u\cdot \nabla v=\lambda_{D}(\Omega_\ell)\int_{\Omega_\ell}|u|^{p-2}uv,\ \forall v\in W_0^{1,p}(\Omega_\ell).
    \end{cases}
\end{equation}
\begin{equation}\label{cross-weak}
    \begin{cases}u\in W_0^{1,p}(\omega),\\
    \int_{\omega}|A_{22}\nabla_{X_2} u\cdot \nabla_{X_2} u|^{\frac{p-2}{2}}A\nabla_{X_2} u\cdot \nabla_{X_2} v=\mu(\omega)\int_{\omega}|u|^{p-2}uv,\ \forall v\in W_0^{1,p}(\omega).
    \end{cases}
\end{equation}
The following characterization of the first eigenvalues $\lambda_{D}^1(\Omega_\ell)$ of \eqref{dir-weak} and  $\mu^1(\omega)$ of \eqref{cross-weak} is well known:
\begin{equation}\label{2.6}
\lambda_{D}^1(\Omega_\ell)=\inf_{u\in W_0^{1,p}(\Omega_\ell),u\neq 0}\frac{\int_{\Omega_\ell}|A\nabla u\cdot \nabla u|^{\frac{p}{2}}}{\int_{\Omega_\ell}|u|^p}.
\end{equation}
\begin{equation}\label{2.7}
    \mu^1(\omega)=\inf_{u\in W_0^{1,p}(\omega),u\neq 0}\frac{\int_{\omega}|A_{22}\nabla_{X_2} u\cdot \nabla_{X_2} u|^{\frac{p}{2}}}{\int_{\omega}|u|^{p}}=\frac{\int_{\omega}|A_{22}\nabla_{X_2} W\cdot \nabla_{X_2} W|^{\frac{p}{2}}}{\int_{\omega}|W|^{p}}.
\end{equation}
It is also a known fact (see \cite{AnLe,Peter}) that both $\lambda_{D}^1(\Omega_\ell)$ and $\mu^1(\omega)$ are simple and the corresponding eigenfunctions do not change sign. Here $W$ is the first normalised non-negative eigenfunction of \eqref{cross-weak}, considered earlier in Theorem \ref{theorem itai and firoj}.
\smallskip

The corresponding weak form of the mixed boundary value problem \eqref{mix} is the following: 
\begin{equation}\label{mix-weak}
    \begin{cases}u\in V(\Omega_\ell),\\
    \int_{\Omega_\ell}|A\nabla u\cdot \nabla u|^{\frac{p-2}{2}}A\nabla u\cdot \nabla v=\lambda_{M}(\Omega_\ell)\int_{\Omega_\ell}|u|^{p-2}uv,\ \forall v\in V(\Omega_\ell)
    \end{cases}
\end{equation}
and the first eigenvalue is given in terms of the Rayleigh quotient as
\begin{equation}\label{2.9}
\lambda_{M}^1(\Omega_\ell)=\inf_{u\in V(\Omega_\ell),u\neq 0}\frac{\int_{\Omega_\ell}|A\nabla u\cdot \nabla u|^{\frac{p}{2}}}{\int_{\Omega_\ell}|u|^p}
\end{equation}
In this case as well, the eigenvalue $\lambda_{M}^1(\Omega_\ell)$ is simple and the corresponding eigenfunction $u_\ell$ have constant sign.

\smallskip

\section{Main result on the asymptotic behavior of the first eigenfunction}
\begin{proof}[\textbf{Proof of Theorem} \ref{theorem decay}]
Let $0<\ell'\leq \ell-1$. Define the function $\rho_{\ell'}=\rho_{\ell'}(x_1)$ by
\begin{equation*}
\rho_{\ell'}(x_1)=
\begin{cases}
  1
     & \text{if $|x_1|\leq \ell',$}\\
  \ell'+1-|x_1|
     & \text{if $|x_1|\in (\ell',\ell'+1),$}\\
  0
     & \text{if $|x_1|\geq \ell'+1.$}
\end{cases}
\end{equation*}
Clearly, $\rho_{\ell'}u_\ell\in W_0^{1,p}(\Omega_\ell)$. From \eqref{2.6} and using the fact that $A$ is symmetric, we get
\begin{multline*}   \lambda_D^1(\Omega_\ell)\int_{\Omega_\ell}\rho_{\ell'}^p|u_\ell|^p\leq \int_{\Omega_\ell}|A\nabla (\rho_{\ell'}u_\ell)\cdot \nabla (\rho_{\ell'}u_\ell)|^{\frac{p}{2}}\\
=\int_{\Omega_\ell }|\rho_{\ell '}^2A\nabla u_\ell \cdot \nabla u_\ell +2\rho_{\ell '} u_\ell A\nabla u_\ell \cdot \nabla \rho_{\ell '}+u_\ell ^2A\nabla \rho_{\ell '}\cdot\nabla \rho_{\ell '}|^{\frac{p}{2}}\\
=\int_{\Omega_\ell }|\rho_{\ell '}^2A\nabla u_\ell \cdot \nabla u_\ell +A\nabla (\rho_{\ell '}u_\ell ^2)\cdot \nabla \rho_{\ell '}|^{\frac{p}{2}}.
 \end{multline*}
 On R.H.S we use the inequality $(a+b)^q\leq a^q+q2^{q-1}(b^q+a^{q-1}b)$ for $a,b\geq0$, $q\geq1$, which gives
\begin{multline}\label{3.1}
\lambda_D^1(\Omega_\ell)\int_{\Omega_\ell}\rho_{\ell'}^p|u_\ell|^p\leq
     \int_{\Omega_\ell }|A\nabla u_\ell \cdot \nabla u_\ell |^{\frac{p}{2}}\rho_{\ell '}^p\\
    +
    \frac{p}{2}C_0\int_{\Omega_\ell }\left(|A\nabla (\rho_{\ell '}u_\ell ^2)\cdot \nabla \rho_{\ell '}|^\frac{p}{2}+|A\nabla u_\ell \cdot \nabla u_\ell |^{\frac{p-2}{2}}|A\nabla (\rho_{\ell '}u_\ell ^2)\cdot \nabla \rho_{\ell '}|\right),
\end{multline}
where $C_0=2^{\frac{p-2}{2}}$. We estimate the integrals on R.H.S separately.
\begin{multline*}
\int_{\Omega_\ell }|A\nabla u_\ell \cdot \nabla u_\ell |^{\frac{p}{2}}\rho_{\ell '}^p\\
=\int_{\Omega_\ell }|A\nabla u_\ell \cdot \nabla u_\ell |^{\frac{p-2}{2}}\left(\rho_{\ell '}^pA\nabla u_\ell \cdot \nabla u_\ell +p\rho_{\ell '}^{p-1} u_\ell A\nabla u_\ell \cdot \nabla \rho_{\ell '}-p\rho_{\ell '}^{p-1} u_\ell A\nabla u_\ell \cdot \nabla \rho_{\ell '}\right)\\
=\int_{\Omega_\ell }|A\nabla u_\ell \cdot \nabla u_\ell |^{\frac{p-2}{2}}A\nabla u_\ell \cdot \nabla (\rho_{\ell '}^p u_\ell )-p\int_{\Omega_\ell }\rho_{\ell '}^{p-1} u_\ell |A\nabla u_\ell \cdot \nabla u_\ell |^{\frac{p-2}{2}}A\nabla u_\ell \cdot \nabla \rho_{\ell '}
\end{multline*}
Using the fact that $u_\ell$ is the first eigenfunction of \eqref{mix-weak} and taking modulus in the last integral, we get
\begin{equation}\label{3.2}
    \int_{\Omega_\ell }|A\nabla u_\ell \cdot \nabla u_\ell |^{\frac{p}{2}}\rho_{\ell '}^p\leq \lambda_M^1(\Omega_\ell )\int_{\Omega_\ell }\rho_{\ell '}^p|u_\ell |^p+p\int_{\Omega_\ell }\rho_{\ell '}^{p-1}|A\nabla u_\ell \cdot \nabla u_\ell |^{\frac{p-2}{2}}|A\nabla u_\ell \cdot \nabla \rho_{\ell '}||u_\ell|.
\end{equation}
Now we estimate the second integral on R.H.S of \eqref{3.1}.
\begin{multline}\label{3.3}
    \int_{\Omega_\ell }|A\nabla (\rho_{\ell '}u_\ell ^2)\cdot \nabla \rho_{\ell '}|^\frac{p}{2}
    =\int_{\Omega_\ell }|2\rho_{\ell '} u_\ell A\nabla u_\ell \cdot \nabla \rho_{\ell '}+u_\ell ^2A\nabla \rho_{\ell '}\cdot\nabla \rho_{\ell '}|^{\frac{p}{2}}\\
    \leq 2^{\frac{p}{2}}C_0\int_{\Omega_\ell }|A\nabla u_\ell \cdot \nabla \rho_{\ell '}|^{\frac{p}{2}}\rho_{\ell '}^{\frac{p}{2}} |u_\ell |^{\frac{p}{2}}+C_0\int_{\Omega_\ell }|A\nabla \rho_{\ell '}\cdot \nabla \rho_{\ell '}|^{\frac{p}{2}} |u_\ell |^{p}.
\end{multline}
Finally the last integral on R.H.S of \eqref{3.1} gives
\begin{multline}\label{3.4}
\int_{\Omega_\ell }|A\nabla u_\ell \cdot \nabla u_\ell |^{\frac{p-2}{2}}|A\nabla (\rho_{\ell '}u_\ell ^2)\cdot \nabla \rho_{\ell '}|\\
=\int_{\Omega_\ell }|A\nabla u_\ell \cdot \nabla u_\ell |^{\frac{p-2}{2}}|2\rho_{\ell '} u_\ell A\nabla u_\ell \cdot \nabla \rho_{\ell '}+u_\ell ^2A\nabla \rho_{\ell '}\cdot\nabla \rho_{\ell '}|
\\
\leq 2\int_{\Omega_\ell }\rho_{\ell '} |A\nabla u_\ell \cdot \nabla u_\ell |^{\frac{p-2}{2}}|A\nabla u_\ell \cdot \nabla \rho_{\ell '}||u_\ell |+\int_{\Omega_\ell }u_\ell^2|A\nabla u_\ell \cdot \nabla u_\ell |^{\frac{p-2}{2}}|A\nabla \rho_{\ell '}\cdot\nabla \rho_{\ell '}|.
\end{multline}
Combining \eqref{3.1}, \eqref{3.2}, \eqref{3.3} and \eqref{3.4} and using $0\leq\rho_{\ell '}\leq1$ we obtain
\begin{multline}\label{3.5}
\left(\lambda_D^1(\Omega_\ell )-\lambda_M^1(\Omega_\ell )\right)\int_{\Omega_\ell }\rho_{\ell '}^p|u_\ell |^p\leq p\int_{\Omega_\ell }|A\nabla u_\ell \cdot \nabla u_\ell |^{\frac{p-2}{2}}|A\nabla u_\ell \cdot \nabla \rho_{\ell '}||u_\ell |\,+\\
pC_0^3\int_{\Omega_\ell }|A\nabla u_\ell \cdot \nabla \rho_{\ell '}|^{\frac{p}{2}}|u_\ell |^{\frac{p}{2}}\,+\frac{p}{2}C_0^2\int_{\Omega_\ell }|A\nabla \rho_{\ell '}\cdot \nabla \rho_{\ell '}|^{\frac{p}{2}} |u_\ell |^{p}\,+\\
pC_0\int_{\Omega_\ell }|A\nabla u_\ell \cdot \nabla u_\ell |^{\frac{p-2}{2}}|A\nabla u_\ell \cdot \nabla \rho_{\ell '}| |u_\ell |\,+\frac{p}{2}C_0\int_{\Omega_\ell }|A\nabla u_\ell \cdot \nabla u_\ell |^{\frac{p-2}{2}}|A\nabla \rho_{\ell '}\cdot\nabla \rho_{\ell '}|u_\ell ^2
\end{multline}
$$=p(C_1I_1+C_2I_2+C_3I_3+C_4I_4),$$
where
$$I_1=\int_{\Omega_\ell }|A\nabla u_\ell \cdot \nabla u_\ell |^{\frac{p-2}{2}}|A\nabla u_\ell \cdot \nabla \rho_{\ell '}| |u_\ell |,$$
$$I_2=\int_{\Omega_\ell }|A\nabla u_\ell \cdot \nabla \rho_{\ell '}|^{\frac{p}{2}}|u_\ell |^{\frac{p}{2}},$$
$$I_3=\int_{\Omega_\ell }|A\nabla \rho_{\ell '}\cdot \nabla \rho_{\ell '}|^{\frac{p}{2}} |u_\ell |^{p},$$
$$I_4=\int_{\Omega_\ell }|A\nabla u_\ell \cdot \nabla u_\ell |^{\frac{p-2}{2}}|A\nabla \rho_{\ell '}\cdot\nabla \rho_{\ell '}|u_\ell ^2$$
and $C_1=(1+C_0)$, $C_2=C_0^3$, $C_3=\frac{1}{2}C_0^2$ and $C_4=\frac{1}{2}C_0$.\smallskip

Let $D_{\ell '}=\Omega_{\ell '+1}\setminus \Omega_{\ell '}$. Note that $|\nabla\rho_{\ell '}|=\chi_{D_{\ell '}}$. We now estimate the integrals $I_1$, $I_2$, $I_3$ and $I_4$ respectively.\\
\textbf{Estimate for $I_1$:} Using Cauchy-Schwarz inequality, we obtain
$$I_1\leq\int_{\Omega_\ell }|A\nabla u_\ell |^{\frac{p}{2}}|\nabla u_\ell |^{\frac{p-2}{2}} |\nabla \rho_{\ell '}||u_\ell| \leq \lVert A \rVert_{\infty}^{\frac{p}{2}}\int_{D_{\ell '}}|\nabla u_\ell |^{p-1} |u_\ell |.$$
By the H\"older's inequality, and then by the Poincar\'e inequality \eqref{poincare}, we obtain
\begin{equation}\label{3.6}
I_1\leq \lVert A \rVert_{\infty}^{\frac{p}{2}}\left(\int_{D_{\ell '}}|\nabla u_\ell |^p\right)^{\frac{p-1}{p}}\left(\int_{D_{\ell '}}|u_\ell |^p\right)^{\frac{1}{p}}\leq \lVert A \rVert_{\infty}^{\frac{p}{2}}C_p\int_{D_{\ell '}}|\nabla u_\ell |^p.
\end{equation}\\
\textbf{Estimate for $I_2$:} Cauchy-Schwarz inequality gives
$$I_2\leq \int_{\Omega_\ell }|A\nabla u_\ell |^{\frac{p}{2}} |\nabla \rho_{\ell '}|^{\frac{p}{2}}|u_\ell |^{\frac{p}{2}}\leq \lVert A \rVert_{\infty}^{\frac{p}{2}}\int_{D_{\ell '}}|\nabla u_\ell |^{\frac{p}{2}} |u_\ell |^{\frac{p}{2}}.$$
Using \eqref{uniform bound} and then the H\"older's inequality and the Poincar\'e inequality \eqref{poincare}, we obtain
\begin{equation}\label{3.7}
    I_2 \leq \lVert A \rVert_{\infty}^{\frac{p}{2}}\left(\int_{D_{\ell '}}|\nabla u_\ell |^p\right)^{\frac{1}{2}}\left(\int_{D_{\ell '}}|u_\ell |^p\right)^{\frac{1}{2}}\leq \lVert A \rVert_{\infty}^{\frac{p}{2}}C_p^{\frac{p}{2}}\int_{D_{\ell '}}|\nabla u_\ell |^p.
\end{equation}\\
\textbf{Estimate for $I_3$:} Using the Poincar\'e inequality \eqref{poincare}, we estimate the the integral
\begin{equation}\label{3.8}
    I_3\leq \lVert A \rVert_{\infty}^{\frac{p}{2}}\int_{D_{\ell '}}| u_\ell |^p\leq \lVert A \rVert_{\infty}^{\frac{p}{2}}C_p^p\int_{D_{\ell '}}|\nabla u_\ell |^p.
\end{equation}\\
\textbf{Estimate for $I_4$:} Similarly we estimate the integral $I_4$ as follows.
$$I_4\leq \lVert A \rVert_{\infty}^{\frac{p}{2}}\int_{D_{\ell '}}|\nabla u_\ell |^{p-2} |u_\ell |^2.$$
Again we use here H\"older's inequality, and then the Poincar\'e inequality \eqref{poincare} to obtain
\begin{equation}\label{3.9}
   I_4\leq \lVert A \rVert_{\infty}^{\frac{p}{2}}\left(\int_{D_{\ell '}}|\nabla u_\ell |^p\right)^{\frac{p-2}{p}}\left(\int_{D_{\ell '}}|u_\ell |^p\right)^{\frac{2}{p}}\leq \lVert A \rVert_{\infty}^{\frac{p}{2}}C_p^2\int_{D_{\ell '}}|\nabla u_\ell |^p.
\end{equation}
Substituting \eqref{3.6}, \eqref{3.7}, \eqref{3.8} and \eqref{3.9} in \eqref{3.5}, we get
\begin{equation}\label{3.10}
    \left(\lambda_D^1(\Omega_\ell )-\lambda_M^1(\Omega_\ell )\right)\int_{\Omega_\ell }\rho_{\ell '}^p|u_\ell |^p\leq p\lVert A \rVert_{\infty}^{\frac{p}{2}} C'\int_{D_{\ell '}}|\nabla u_\ell |^p
\end{equation}
where $C'=\left(C_1C_p+C_2C_p^{\frac{p}{2}}+C_3C_p^p+C_4C_p^2\right)$. Notice that,
\eqref{result firoj 2} together with \eqref{result firoj 1} gives 
a $\beta'>0$ such that for $\ell >\ell _0$ we have $\lambda_D^1(\Omega_\ell )-\lambda_M^1(\Omega_\ell )>\beta'$ and \eqref{3.10} gives
\begin{equation}\label{3.11}
 \beta'\int_{\Omega_\ell }\rho_{\ell '}^p|u_\ell |^p\leq p\lVert A \rVert_{\infty}^{\frac{p}{2}} C'\int_{D_{\ell '}}|\nabla u_\ell |^p.
\end{equation}\\
Since $\rho_{\ell '}^pu_\ell\in V(\Omega_\ell)$, from \eqref{mix-weak} we get
\begin{multline*}
    \lambda_M^1(\Omega_\ell )\int_{\Omega_\ell }\rho_{\ell '}^p|u_\ell |^p=\int_{\Omega_\ell }|A\nabla u_\ell \cdot \nabla u_\ell |^{\frac{p-2}{2}}A\nabla u_\ell \cdot \nabla (\rho_{\ell '}^pu_\ell )\\
    =\int_{\Omega_\ell }|A\nabla u_\ell \cdot \nabla u_\ell |^{\frac{p}{2}}\rho_{\ell '}^p+p\int_{\Omega_\ell }\rho_{\ell '}^{p-1}u_\ell |A\nabla u_\ell \cdot \nabla u_\ell |^{\frac{p-2}{2}}A\nabla u_\ell \cdot \nabla \rho_{\ell '}.
\end{multline*}
Using the inequality $a-|b|\leq a+b$ for $a,b\in\RR$, we obtain
$$\lambda_M^1(\Omega_\ell )\int_{\Omega_\ell }\rho_{\ell '}^p|u_\ell |^p\geq \int_{\Omega_\ell }|A\nabla u_\ell \cdot \nabla u_\ell |^{\frac{p}{2}}\rho_{\ell '}^p-p\int_{\Omega_\ell }|A\nabla u_\ell \cdot \nabla u_\ell |^{\frac{p-2}{2}}|A\nabla u_\ell \cdot \nabla \rho_{\ell '}|\rho_{\ell '}^{p-1}|u_\ell |.$$
Using \eqref{uniform ellip} and $\rho_{\ell'}=1$ in $\Omega_{\ell '}$ in the first integral, $0\leq\rho_{\ell '}\leq1$ and $|\nabla\rho_{\ell '}|=\chi_{D_{\ell '}}$ and then the H\"older's inequality in the second integral of the R.H.S, we get
\begin{multline*}
    \lambda_M^1(\Omega_\ell )\int_{\Omega_\ell }\rho_{\ell '}^p|u_\ell |^p\geq \lambda^{\frac{p}{2}}\int_{\Omega_{\ell '}}|\nabla u_\ell |^p-p\lVert A \rVert_{\infty}^{\frac{p}{2}}\int_{D_{\ell '}}|\nabla u_\ell |^{p-1} |u_\ell |\\
    \geq \lambda^{\frac{p}{2}}\int_{\Omega_{\ell '}}|\nabla u_\ell |^p-p\lVert A \rVert_{\infty}^{\frac{p}{2}}\left(\int_{D_{\ell '}}|\nabla u_\ell |^p\right)^{\frac{p-1}{p}}\left(\int_{D_{\ell '}}|u_\ell |^p\right)^{\frac{1}{p}}.
\end{multline*}
Finally the Poincar\'e inequality \eqref{poincare} on the last integral gives
\begin{equation}\label{3.12}
    \lambda_M^1(\Omega_\ell )\int_{\Omega_\ell }\rho_{\ell '}^p|u_\ell |^p\geq \lambda^{\frac{p}{2}}\int_{\Omega_{\ell '}}|\nabla u_\ell |^p-p\lVert A \rVert_{\infty}^{\frac{p}{2}}C_p\int_{D_{\ell '}}|\nabla u_\ell |^p.
\end{equation}\\
Combining \eqref{3.12} with \eqref{3.11}, we get
$$\beta'\lambda^{\frac{p}{2}}\int_{\Omega_{\ell '}}|\nabla u_\ell |^p\leq p\lVert A \rVert_{\infty}^{\frac{p}{2}}\left(\lambda_M^1(\Omega_\ell )C'+\beta'C_p\right)\int_{D_{\ell '}}|\nabla u_\ell |^p,$$
which gives
\begin{equation}\label{3.13}
    \beta\int_{\Omega_{\ell '}}|\nabla u_\ell |^p\leq C\int_{D_{\ell '}}|\nabla u_\ell |^p,
\end{equation}
where $\beta=\beta'\lambda^{\frac{p}{2}}$, and $C=p\lVert A \rVert_{\infty}^{\frac{p}{2}}(\mu^1(\omega)C'+\beta'C_p)$. Here we use the fact that $\lambda_M^1(\Omega_\ell )\leq \mu^1(\omega)$, which follows from \eqref{2.9} by using $u(x)=W(X_2)\in V(\Omega_\ell)$ as a test function. Therefore, from \eqref{3.13} we deduce that
$$(\beta+C)\int_{\Omega_{\ell '}}|\nabla u_\ell |^p\leq C\int_{\Omega_{\ell '+1}}|\nabla u_\ell |^p,$$
and we have
$$\int_{\Omega_{\ell '}}|\nabla u_\ell |^p\leq \alpha\int_{\Omega_{\ell '+1}}|\nabla u_\ell |^p,$$
where $\alpha=\frac{C}{\beta+C}<1$. Applying this procedure successively for $\ell '=r,r+1,\cdots,r+[\ell -r]-1$, we get
\begin{equation}\label{3.14}
    \int_{\Omega_r}|\nabla u_\ell |^p\leq \alpha^{[\ell -r]}\int_{\Omega_{\ell }}|\nabla u_\ell |^p\leq \frac{\alpha^{[\ell -r]}}{\lambda^{\frac{p}{2}}}\int_{\Omega_\ell }|A\nabla u_\ell \cdot \nabla u_\ell |^{\frac{p}{2}}=\frac{\lambda_M^1(\Omega_\ell )}{\lambda^{\frac{p}{2}}}\alpha^{[\ell -r]}\leq \frac{\mu^1(\omega)}{\lambda^{\frac{p}{2}}}\alpha^{[\ell -r]},
\end{equation}
which proves \eqref{decay of grad}. To prove \eqref{decay of u}, we use the Poincar\'e inequality \eqref{poincare} and \eqref{3.14} and we get
\begin{equation}\label{3.15}
    \int_{\Omega_r}|u_\ell |^p\leq C_p^p\int_{\Omega_r}|\nabla u_\ell |^p\leq\frac{\mu^1(\omega)C_p^p}{\lambda^{\frac{p}{2}}}\alpha^{[\ell -r]}.
\end{equation}
\eqref{3.14} and \eqref{3.15} together completes the proof.
\end{proof}
An immediate corollary of this theorem is the following, which gives the concentration of the masses near the ends of the cylinder.
\begin{corollary}
If $A_{12}\cdot \nabla_{X_2}W\not\equiv0$ a.e. in $\omega$, then for large $\ell$ and for any $0<r\leq\ell-1$, we have
\begin{equation}
\int_{\Omega_\ell\setminus\Omega_r}|\nabla u_\ell|^p\geq 1-C_1\alpha^{[\ell-r]}\quad and \quad \int_{\Omega_\ell\setminus\Omega_r}| u_\ell|^p\geq 1-C_2\alpha^{[\ell-r]}.
\end{equation}
\end{corollary}
\section{ Connection with the problem on semi-infinite cylinder}
Although, unlike the first Dirichlet eigenvalues it is not known whether the function $\ell\rightarrow \lambda_M^1(\Omega_\ell )$ is monotone or not, we show that the limit of $\lambda_M^1(\Omega_\ell )$ exists as $\ell$ goes to infinity (Theorem \ref{theorem limit of first eigenvalue}). To identify the limit we introduce the variational problems \eqref{nu infinity} on semi-infinite cylinders $\Omega_{\infty}^{\pm}$. In next lemma, we find the possible range of the quantities $\nu_{\infty}^{\pm}$ defined in \eqref{nu infinity} and then (in Lemma \ref{lemma 4.2}) we identify $\nu_{\infty}^{\pm}$ with the limits of specific minimization problems (defined later in \eqref{4.4}) defined on the half-cylinders $\Omega_{\ell}^{\pm}$, which we will use to prove Theorem \ref{theorem limit of first eigenvalue}. 
\begin{remark}\label{remark 1}
If $A$ has the symmetry (S), then it is easy to check that $\nu_{\infty}^+=\nu_{\infty}^-$, since $v(-x_1,-X_2)\in V(\Omega_{\infty}^-)$ for any $v(x_1,X_2)\in V(\Omega_{\infty}^+)$ and vice versa, and clearly both give the same value in \eqref{nu infinity}. Otherwise, we use $\nu_{\infty}^-=\tilde{\nu}_{\infty}^+$, where $\tilde{\nu}_{\infty}^+$ is defined similarly as $\nu_{\infty}^+$, but in place of $A(X_2)$ we put 
$$\tilde{A}(X_2)=\begin{pmatrix} a_{11}(X_2) & -A_{12}(X_2)\\ -A_{12}^t(X_2) & A_{22}(X_2) \end{pmatrix}.$$
\end{remark}\smallskip

\begin{lemma}\label{lemma 4.1}
We have
\begin{equation}\label{4.1}
    \lambda C_p^p\leq\nu_{\infty}^{\pm}\leq \mu^1(\omega).
\end{equation}
where $\lambda$ is the constant given in \eqref{uniform ellip} and $C_p$ is the constant of Poincar\'e inequality \eqref{poincare}.
\end{lemma}
\begin{proof}
Since $\tilde{A}_{22}(X_2)=A_{22}(X_2)$ as shown in Remark \ref{remark 1}, and $\mu^1(\omega)$ depends only on $A_{22}(X_2)$, it is enough to prove \eqref{4.1} for $\nu_{\infty}^+$. We get the lower bound $\nu_{\infty}^+\geq \lambda C_p^p$ by using the uniform ellipticity \eqref{uniform ellip} of $A$ and the Poincar\'e inequality \eqref{poincare}. To show the upper bound $\nu_{\infty}^+\leq \mu^1(\omega)$, we first fix $\epsilon>0$. For $x=(x_1,X_2)\in \Omega_{\infty}^+$ define the function $v_{\epsilon}$ in $V(\Omega_{\infty}^+)$ as
$$v_{\epsilon}(x)=e^{-\epsilon x_1}W(X_2).$$
We have
\begin{multline*}
    \int_{\Omega_{\infty}^+}|A\nabla v_{\epsilon}\cdot \nabla v_{\epsilon}|^{\frac{p}{2}}
=\int_{\Omega_{\infty}^+}\left|e^{-2\epsilon x_1}\left(\epsilon^2a_{11}W^2-2\epsilon(A_{12}\cdot \nabla_{X_2}W)W+A_{22}\nabla_{X_2}W\cdot \nabla_{X_2}W\right)\right|^{\frac{p}{2}}\\
\leq \left(\int_0^{\infty}e^{-p\epsilon x_1}\right)\int_{\omega}\left(|A_{22}\nabla_{X_2}W\cdot \nabla_{X_2}W|+\epsilon|\epsilon a_{11} W^2-2(A_{12}\cdot \nabla_{X_2}W)W| \right)^{\frac{p}{2}}.
\end{multline*}
Splitting the integral on $\omega$ by using the inequality $(a+b)^q\leq a^q+q2^{q-1}(a^{q-1}b+b^q)$ for $a,b\geq0$, $q\geq1$, and then using \eqref{2.7} and the fact that  $W$ is normalised we get
\begin{multline}\label{4.2}
    \int_{\Omega_{\infty}^+}|A\nabla v_{\epsilon}\cdot \nabla v_{\epsilon}|^{\frac{p}{2}}
\leq \left(\int_0^{\infty}e^{-p\epsilon x_1}\right)\Big[\mu^1(\omega)+\epsilon C\int_{\omega}|A_{22}\nabla_{X_2}W\cdot \nabla_{X_2}W|^{\frac{p-2}{2}}\\
\times|\epsilon a_{11} W^2-2(A_{12}\cdot \nabla_{X_2}W)W|
+\epsilon^{\frac{p}{2}}C\int_{\omega}|\epsilon a_{11} W^2-2(A_{12}\cdot \nabla_{X_2}W)W|^{\frac{p}{2}}\Big].
\end{multline}
As
\begin{equation}\label{4.3}
    \int_{\Omega_{\infty}^+}|v_{\epsilon}|^p=\int_0^{\infty}e^{-p\epsilon x_1},
\end{equation}
substituting \eqref{4.2}, \eqref{4.3} in \eqref{nu infinity}, we get
\begin{multline*}
    \nu_{\infty}^+\leq \mu^1(\omega)+\epsilon C\int_{\omega}|A_{22}\nabla_{X_2}W_1\cdot \nabla_{X_2}W_1|^{\frac{p-2}{2}}|a_{11}\epsilon W_1^2-2(A_{12}\cdot \nabla_{X_2}W_1)W_1|\\
+\epsilon^{\frac{p}{2}}C\int_{\omega}|a_{11}\epsilon W_1^2-2(A_{12}\cdot \nabla_{X_2}W_1)W_1|^{\frac{p}{2}}.
\end{multline*}
Passing to the limit $\epsilon\rightarrow0$ we get the result.
\end{proof}\smallskip

We now consider the minimization problems
\begin{equation}\label{4.4}
    \tilde{\lambda}_{M}^1(\Omega_\ell ^{\pm})=\inf \left\{\int_{\Omega_\ell ^{\pm}}|A\nabla u\cdot \nabla u|^{\frac{p}{2}}:u\in W^{1,p}(\Omega_\ell ^{\pm}),\int_{\Omega_\ell ^{\pm}}|u|^p=1, u=0 \ \textrm{on} \ \gamma_\ell ^{\pm}\cup \Gamma_\ell ^{\pm}\right\}.
\end{equation}
in the half cylinders
\begin{equation}\label{4.5}
    \Omega_\ell ^+=(0,\ell )\times\omega\ \mathrm{and}\ \Omega_\ell ^-=(-\ell ,0)\times\omega,
\end{equation}
where the boundaries are defined by
\begin{equation}\label{4.6}
    \Gamma_\ell ^{\pm}=\{\pm \ell \}\times\omega,\ \gamma_\ell ^+=(0,\ell )\times\partial\omega\ \textrm{and}\ \gamma_\ell ^-=(-\ell ,0)\times\partial\omega.
\end{equation}
 In the following lemma, we identify the values $\nu_{\infty}^{\pm}$ with the limits (as $\ell$ goes to infinity) of the quantities defined in \eqref{4.4}.
\begin{remark}\label{remark 2}
Since for any fixed $\ell>0$ the domains are bounded, it is a well known fact that the infimum in \eqref{4.4} is attained. We denote by $\tilde{u}_\ell ^{\pm}$, the positive normalized minimizers corresponding to $\tilde{\lambda}_{M}^1(\Omega_\ell ^{\pm})$.
\end{remark}
\begin{lemma}\label{lemma 4.2}
We have
\begin{equation}\label{4.7}
    \nu_{\infty}^{\pm}=\lim_{\ell \rightarrow\infty}\tilde{\lambda}_{M}^1(\Omega_\ell ^{\pm}).
\end{equation}
\end{lemma}
\begin{proof}
It is enough to prove \eqref{4.7} only for $\tilde{\lambda}_{M}^1(\Omega_{\ell} ^+)$, proof for $\tilde{\lambda}_{M}^1(\Omega_{\ell} ^-)$ is similar. 
If $\ell_1<\ell_2$, then any admissible function in \eqref{4.4} for $\tilde{\lambda}_{M}^1(\Omega_{\ell_1}^+)$ can be extended to an admissible function for $\tilde{\lambda}_{M}^1(\Omega_{\ell_2}^+)$ by setting it zero on $\Omega_{\ell _2}^+\setminus\Omega_{\ell _1}^+$, which shows that $\tilde{\lambda}_{M}^1(\Omega_{\ell_1}^+)\geq \tilde{\lambda}_{M}^1(\Omega_{\ell_2}^+).$ Therefore, monotonicity of the function $\ell\mapsto\tilde{\lambda}_{M}^1(\Omega_{\ell} ^+)$ implies that $\lim_{\ell \rightarrow\infty}\tilde{\lambda}_{M}^1(\Omega_{\ell} ^+)$ exists.\\
A similar argument shows 
that $\tilde{\lambda}_{M}^1(\Omega_{\ell }^+)\geq\nu_{\infty}^+$, $\forall$ $\ell >0$.\smallskip

The following space
\begin{equation}\label{Vs space}
V_s(\Omega_{\infty}^+)=\left\{u\in C^{\infty}(\Omega_{\infty}^+)\cap V(\Omega_{\infty}^+):\exists M=M(u)>0 \ \textrm{such that} \ u=0 \ \textrm{on} \ (M,\infty)\times\omega \right\}
\end{equation}
is clearly dense in $V(\Omega_{\infty}^+)$. Let $\Tilde{u}=\frac{u}{\lVert u \rVert_{L^p}}$, where $u\in V(\Omega_{\infty}^+)\setminus \{0\}$. Then $\exists\{v_n\}\in V_s(\Omega_{\infty}^+)\setminus \{0\}$ with supp$(v_n)\subset V(\Omega_{\ell_n}^+)$ such that $v_n\rightarrow \Tilde{u}$ in $V(\Omega_{\infty}^+)$ as $n\rightarrow\infty$. In particular $\{v_n\}$ is bounded in $V(\Omega_{\infty}^+)$ and $\lVert v_n \rVert_{L^p(\Omega_{\infty}^+)}\rightarrow1$.\smallskip

Define $\Tilde{v_n}=\frac{v_n}{\lVert v_n \rVert_{L^p}}$ and let $\epsilon>0$. Then for all $n\geq n_0$,
\begin{multline}
    \label{4.9}
    \lVert\Tilde{u}-\Tilde{v_n}\rVert_{W^{1,p}(\Omega_{\infty}^+)}\leq \lVert \Tilde{u}-v_n\rVert_{W^{1,p}(\Omega_{\infty}^+)}+\lVert v_n-\frac{v_n}{\lVert v_n \rVert_{L^p}}\rVert_{W^{1,p}(\Omega_{\infty}^+)}\\
    = \lVert\Tilde{u}-v_n\rVert_{W^{1,p}(\Omega_{\infty}^+)}+\left|1-\frac{1}{\lVert v_n \rVert_{L^p}}\right|\lVert v_n \rVert_{W^{1,p}(\Omega_{\infty}^+)}<\epsilon.
\end{multline}
Using the inequality $|a^p-b^p|\leq2p|a-b|\mathrm{max}\left\{a^{p-1},b^{p-1}\right\}$ for $a,b\geq0$ and $p\geq1$ and then Holder's inequality we obtain
\begin{multline}\label{4.10}
I_n=\left|\int_{\Omega_{\infty}^+}\left(|A\nabla \Tilde{u}\cdot \nabla \Tilde{u}|^{\frac{p}{2}}-|A\nabla \Tilde{v_n}\cdot \nabla \Tilde{v_n}|^{\frac{p}{2}}\right)\right|\\
\leq p\int_{\Omega_{\infty}^+}|A\nabla \Tilde{u}\cdot \nabla \Tilde{u}-A\nabla \Tilde{v_n}\cdot \nabla \Tilde{v_n}|\mathrm{max}\left\{|A\nabla \Tilde{u}\cdot \nabla \Tilde{u}|^{\frac{p-2}{2}},|A\nabla \Tilde{v_n}\cdot \nabla \Tilde{v_n}|^{\frac{p-2}{2}}\right\}\\
 \leq p\left(\int_{\Omega_{\infty}^+}|A\nabla \Tilde{u}\cdot \nabla \Tilde{u}-A\nabla \Tilde{v_n}\cdot \nabla \Tilde{v_n}|^{\frac{p}{2}}\right)^{\frac{2}{p}}\left(\int_{\Omega_{\infty}^+}\mathrm{max}\left\{|A\nabla \Tilde{u}\cdot \nabla \Tilde{u}|^{\frac{p}{2}},|A\nabla \Tilde{v_n}\cdot \nabla \Tilde{v_n}|^{\frac{p}{2}}\right\}\right)^{\frac{p-2}{p}}
\end{multline}
The second integral in R.H.S of \eqref{4.10} is clearly bounded. Using the identity $A\nabla u\cdot \nabla u-A\nabla v\cdot \nabla v=A\nabla (u-v)\cdot\nabla (u-v)+2A\nabla (u-v)\cdot\nabla v$ in \eqref{4.10}, we get
\begin{multline}\label{4.11}
    I_n\leq pC_1\left(\int_{\Omega_{\infty}^+}|A\nabla (\Tilde{u}-\Tilde{v_n})\cdot\nabla (\Tilde{u}-\Tilde{v_n})+2A\nabla (\Tilde{u}-\Tilde{v_n})\cdot\nabla \Tilde{v_n}|^{\frac{p}{2}}\right)^{\frac{2}{p}}\\
    \leq pC_1\left(2^{\frac{p-2}{2}}\int_{\Omega_{\infty}^+}|A\nabla (\Tilde{u}-\Tilde{v_n})\cdot\nabla (\Tilde{u}-\Tilde{v_n})|^{\frac{p}{2}}+2^{p-1}\int_{\Omega_{\infty}^+}|A\nabla (\Tilde{u}-\Tilde{v_n})\cdot\nabla \Tilde{v_n}|^{\frac{p}{2}}\right)^{\frac{2}{p}}
\end{multline}
Using \eqref{4.9}, we estimate the first integral of \eqref{4.11} as
\begin{equation}\label{4.12}
\int_{\Omega_{\infty}^+}|A\nabla (\Tilde{u}-\Tilde{v_n})\cdot\nabla (\Tilde{u}-\Tilde{v_n})|^{\frac{p}{2}}\leq \lVert A \rVert_{\infty}^{\frac{p}{2}}\epsilon^p.
\end{equation}
H\"older's inequality together with \eqref{4.9} gives
\begin{equation}\label{4.13}
\int_{\Omega_{\infty}^+}|A\nabla (\Tilde{u}-\Tilde{v_n})\cdot\nabla \Tilde{v_n}|^{\frac{p}{2}}\leq C_2\lVert A \rVert_{\infty}^{\frac{p}{2}}\epsilon^{\frac{p}{2}}.
\end{equation}
Combining \eqref{4.12}, \eqref{4.13} with \eqref{4.11} we get for $n\geq n_0$,
\begin{equation}\label{4.14}
I_n=\left|\frac{\int_{\Omega_{\infty}^+}|A\nabla u\cdot \nabla u|^{\frac{p}{2}}}{\int_{\Omega_{\infty}^+}|u|^p}-\frac{\int_{\Omega_{\infty}^+}|A\nabla v_n\cdot \nabla v_n|^{\frac{p}{2}}}{\int_{\Omega_{\infty}^+}|v_n|^p}\right|\leq C\epsilon.
\end{equation}
This proves \eqref{4.7}.

\end{proof}\smallskip

\begin{proof}[\textbf{Proof of Theorem} \ref{theorem limit of first eigenvalue}]
(i) First we show that
\begin{equation}\label{4.15}
    \limsup_{\ell \rightarrow\infty}\lambda_M^1(\Omega_\ell )\leq\min \left\{\nu_{\infty}^+,\nu_{\infty}^-\right\}.
\end{equation}
Without loss of generality we may assume that $\nu_{\infty}^+=\min \{\nu_{\infty}^+,\nu_{\infty}^-\}.$
For any $\ell>0$ the function $\Tilde{u}_\ell^+(x_1-\frac{\ell}{2},X_2)\in V(\Omega_{\ell /2})$, where $\Tilde{u}_\ell^+(x_1,X_2)$ is the minimizer for $\tilde{\lambda}_{M}^1(\Omega_{\ell }^+)$ (see Remark \ref{remark 2}), is a suitable test function in \eqref{2.9} and we have $\lambda_M^1(\Omega_{\ell /2})\leq\tilde{\lambda}_{M}^1(\Omega_{\ell }^+)$. On the other hand, for any $\epsilon>0$, Lemma \ref{lemma 4.2} gives an $\ell_{\epsilon}>0$ such that $\tilde{\lambda}_{M}^1(\Omega_{2\ell _{\epsilon}}^+)<\nu_{\infty}^++\epsilon$. Monotonicity of $\ell\mapsto\tilde{\lambda}_{M}^1(\Omega_{\ell }^+)$ gives $\sup_{\ell\geq \ell_{\epsilon}}\lambda_M^1(\Omega_{\ell})\leq\nu_{\infty}^++\epsilon$ and \eqref{4.15} holds.\smallskip

(ii) For the reverse inequality first we consider the case $A_{12}\cdot\nabla W\not\equiv0$ a.e. in $\omega$. Let $\rho(x_1)$ be the function defined by
$$\rho(x_1)=
\begin{cases}
  0
     & \text{if $x_1\leq -1,$}\\
  1+x_1
     & \text{if $x_1\in (-1,0),$}\\
  1
     & \text{if $x_1\geq 0.$}
\end{cases}$$
Define $w_{\ell +1}(x_1,X_2)=\rho(x_1+\ell)u_\ell (x_1+\ell ,X_2)$ on $\Omega_{\ell +1}^-$, where $u_\ell$ is the minimizer in \eqref{2.9}. It is easy to see that $w_{\ell +1}$ is a suitable test function for defining $\Tilde{\lambda}_M^1(\Omega_{\ell +1}^-)$ in \eqref{4.4}. Note that $|\nabla\rho|=1$ in $(-1,0)\times\omega$. Hence
\begin{multline}\label{4.16}
    \int_{\Omega_{\ell +1}^-}|A\nabla w_{\ell +1}\cdot \nabla w_{\ell +1}|^{\frac{p}{2}}\leq \int_{\Omega_\ell ^+}|A\nabla u_\ell \cdot \nabla u_\ell |^{\frac{p}{2}}+\int_{(-1,0)\times\omega}|A\nabla(\rho u_\ell )\cdot \nabla(\rho u_\ell )|^{\frac{p}{2}}\\
    \leq \int_{\Omega_\ell ^+}|A\nabla u_\ell \cdot \nabla u_\ell |^{\frac{p}{2}}+2^{p-1}\lVert A \rVert_{\infty}^{\frac{p}{2}} \int_{(-1,0)\times\omega}\left(|u_\ell|^p+|\nabla u_\ell|^p\right).
\end{multline}
Substituting \eqref{decay of grad}, \eqref{decay of u} in \eqref{4.16}, we get a constant $C>0$ such that
\begin{equation}\label{4.17}
 \int_{\Omega_{\ell +1}^-}|A\nabla w_{\ell +1}\cdot \nabla w_{\ell +1}|^{\frac{p}{2}}\leq \int_{\Omega_\ell ^+}|A\nabla u_\ell \cdot \nabla u_\ell |^{\frac{p}{2}}+C\alpha^{[\ell -1]}.
\end{equation}
We denote by $N_\ell ^{\pm}$ and $D_\ell ^{\pm}$ the quantities
$$N_\ell ^{\pm}=\int_{\Omega_\ell ^{\pm}}|A\nabla u_\ell \cdot \nabla u_\ell |^{\frac{p}{2}},\quad D_\ell ^{\pm}=\int_{\Omega_\ell ^{\pm}}|u_\ell |^p.$$
Clearly
\begin{equation}\label{4.18}
    N_\ell ^++N_\ell ^-=\lambda_M^1(\Omega_\ell )\quad and\quad D_\ell ^++D_\ell ^-=1.
\end{equation}
By the same computation for $\Omega_{\ell +1}^+$ and then using \eqref{4.18} we get
\begin{equation}\label{4.19}
    \tilde{\lambda}_{M}^1(\Omega_{\ell +1}^-)\leq \frac{N_\ell ^++C\alpha^{[\ell -1]}}{D_\ell ^+}\ \mathrm{and} \ \tilde{\lambda}_{M}^1(\Omega_{\ell +1}^+)\leq \frac{N_\ell ^-+C\alpha^{[\ell -1]}}{D_\ell ^-}.
\end{equation}
By \eqref{4.19} and \eqref{4.18}, we have
\begin{equation}\label{4.20}
    \min\{\tilde{\lambda}_{M}^1(\Omega_{\ell +1}^-),\tilde{\lambda}_{M}^1(\Omega_{\ell +1}^+)\}\leq D_\ell ^+\tilde{\lambda}_{M}^1(\Omega_{\ell +1}^-)+D_\ell ^-\tilde{\lambda}_{M}^1(\Omega_{\ell +1}^+)\leq \lambda_M^1(\Omega_\ell )+2C\alpha^{[\ell -1]}.
\end{equation}
Passing to the limit as $\ell\rightarrow\infty$ on both sides, and using Lemma \ref{lemma 4.2} it follows that
\begin{equation}\label{4.21}
    \min \{\nu_{\infty}^+,\nu_{\infty}^-\}\leq \liminf_{\ell \rightarrow\infty}\lambda_M^1(\Omega_\ell ),
\end{equation}
which together with \eqref{4.15} proves the theorem in this case.
\smallskip

(iii) We now consider the case $A_{12}\cdot\nabla W\equiv0$. In this case it is already known from Theorem \ref{theorem itai and firoj} that $\lambda_M^1(\Omega_\ell )=\mu^1(\omega)$ for all $\ell>0$. On the other hand we see in Theorem 
\ref{theorem gap holds for mu} (ii) that $\nu_{\infty}^{\pm}=\mu^1(\omega)$ in this case. This completes the proof.
\end{proof}\smallskip

In the next proposition we show that the first eigenfunction for the problem \eqref{mix} on $\Omega_\ell $ decays to $0$ on that semi-infinite cylinder which corresponds to the strictly smaller of the two values  $\nu_{\infty}^+$ and $\nu_{\infty}^-.
$

\begin{proposition}\label{proposition 4.1}
If $\nu_{\infty}^+<\nu_{\infty}^-$, then $\lim_{\ell \rightarrow \infty}\int_{\Omega_\ell ^+}|\nabla u_\ell |^p+|u_\ell |^p=0$.
\end{proposition}
\begin{proof}
Taking the limit $\ell \rightarrow\infty$ in \eqref{4.20} and using  \eqref{4.7} and \eqref{limit of first eigenvalue}, we obtain
$$\left(\limsup_{\ell \rightarrow\infty}D_\ell ^+\right)\nu_{\infty}^-+\left(1-\limsup_{\ell \rightarrow\infty}D_\ell ^+\right)\nu_{\infty}^+\leq \nu_{\infty}^+$$
which gives
$$\limsup_{\ell \rightarrow\infty}D_\ell ^+(\nu_{\infty}^--\nu_{\infty}^+) \leq 0,$$
and we must have $\limsup_{\ell \rightarrow\infty}D_\ell ^+=0$. Again by \eqref{4.19} we have,
\begin{multline}\label{4.22}
    \frac{N_\ell ^+}{D_\ell ^-}+\tilde{\lambda}_{M}^1(\Omega_{\ell +1}^+)-\frac{C\alpha^{[\ell -1]}}{D_\ell ^-}\leq \frac{N_\ell ^++N_\ell ^-}{D_\ell ^-} \\
    =\frac{N_\ell ^++N_\ell ^-}{D_\ell ^++D_\ell ^-}+\frac{(N_\ell ^++N_\ell ^-)D_\ell ^+}{(D_\ell ^++D_\ell ^-)D_\ell ^-}=\lambda_M^1(\Omega_\ell )+\frac{D_\ell ^+}{D_\ell ^-}(N_\ell ^++N_\ell ^-).
\end{multline}
Since $\lim_{\ell \rightarrow\infty}D_\ell ^-=1$, passing to the limit $\ell \rightarrow\infty$ in \eqref{4.22} and using \eqref{4.7} and \eqref{limit of first eigenvalue}, we get $\lim_{\ell \rightarrow\infty}N_\ell ^+= 0$. Finally $\lim_{\ell \rightarrow\infty}\int_{\Omega_\ell ^+}|\nabla u_\ell |^p=0$ follows using \eqref{uniform ellip}.
\end{proof}
\section{Gap phenomenon on semi-infinite cylinder}
In this section we revisit the problem \eqref{nu infinity} on the semi-infinite cylinders. We first investigate  when the infimum in \eqref{nu infinity} is attained and follow it up  with Theorem \ref{theorem gap holds for mu} where we investigate when the gap between $\nu_{\infty}^{\pm}$ and $\mu^1(\omega)$ holds and when does not hold.
\begin{proposition}\label{proposition 5.1}
If
\begin{equation}\label{5.1}
    \nu_{\infty}^+< \mu^1(\omega),
\end{equation} 
then $\nu_{\infty}^+$ is attained. The minimizer $\tilde{u}^+$ in \eqref{nu infinity} is unique upto multiplication by a constant, has constant sign and satisfies
\begin{equation}\label{5.2}
    \begin{cases}
    -\mathrm{div}(|A\nabla \tilde{u}^+\cdot \nabla \tilde{u}^+|^{\frac{p-2}{2}}A\nabla \tilde{u}^+)=\nu_{\infty}^+|\tilde{u}^+|^{p-2}\tilde{u}^+\ \mathrm{in} \ \Omega_{\infty}^+,\\
    \tilde{u}^+=0\ \mathrm{on}\ \gamma_{\infty}^+,\\
    (A\nabla \tilde{u}^+)\cdot \nu=0\ \mathrm{on}\ \{0\}\times \omega.
    \end{cases}
\end{equation}
\end{proposition}
\begin{proof}
For any $\ell>0$ the function $\tilde{u}_\ell ^+$ (defined in Remark \ref{remark 2}) can be extended by zero to a function in $W^{1,p}(\Omega_{\infty}^+)$. Since the sequence $\{\tilde{u}_\ell ^+\}$ is bounded in $W^{1,p}(\Omega_{\infty}^+)$, there is a subsequence $\{\tilde{u}_{\ell _k}^+\}$ which converges weakly to a limit $\tilde{u}^+\in W^{1,p}(\Omega_{\infty}^+)$. Definition of $\tilde{\lambda}_M^1(\Omega_{\ell _k}^+)$ gives the following identity:
\begin{equation}\label{5.3}
    \int_{\Omega_{\infty}^+}|A\nabla \tilde{u}_{\ell _k}^+\cdot \nabla \tilde{u}_{\ell _k}^+|^{\frac{p}{2}}=\tilde{\lambda}_M^1(\Omega_{\ell _k}^+)\int_{\Omega_{\infty}^+}|\tilde{u}_{\ell _k}^+|^p.
\end{equation}
We claim that
\begin{equation}\label{5.4}
    \int_{\Omega_{\infty}^+}|A\nabla \tilde{u}^+\cdot \nabla \tilde{u}^+|^{\frac{p}{2}}=\nu_{\infty}^+
    \int_{\Omega_{\infty}^+}|\tilde{u}^+|^p.
\end{equation}
Clearly $\tilde{u}^+\in V(\Omega_{\infty}^+)$. From \eqref{nu infinity} we have
\begin{equation}\label{5.5}
\int_{\Omega_{\infty}^+}|A\nabla \tilde{u}^+\cdot \nabla \tilde{u}^+|^{\frac{p}{2}}\geq\nu_{\infty}^+
    \int_{\Omega_{\infty}^+}|\tilde{u}^+|^p.
\end{equation}
By Fatou's lemma and by \eqref{5.3}, we get
\begin{equation}\label{5.6}
    \int_{\Omega_{\infty}^+}|A\nabla \tilde{u}^+\cdot \nabla \tilde{u}^+|^{\frac{p}{2}}\leq \liminf_{k\rightarrow\infty} \int_{\Omega_{\infty}^+}|A\nabla \tilde{u}_{\ell _k}^+\cdot \nabla \tilde{u}_{\ell _k}^+|^{\frac{p}{2}}=\nu_{\infty}^+ \left(\liminf_{k\rightarrow\infty}\int_{\Omega_{\infty}^+}|\tilde{u}_{\ell _k}^+|^p\right)=\nu_{\infty}^+.
\end{equation}
To prove our claim, it is enough to show that $\int_{\Omega_{\infty}^+}|\tilde{u}^+|^p=1$. In addition we will show that $\tilde{u}_\ell ^+$ decays to zero for large $x_1$, and that implies concentration near $x_1=0$, using the same technique as in the proof of Theorem \ref{theorem decay}.\smallskip

Let $\ell $ and $\ell '$ be such that $0<\ell '\leq \ell -1$. Define $\tilde{\rho}_{\ell '}=\tilde{\rho}_{\ell '}(x_1)$ by

$$\tilde{\rho}_{\ell '}(x_1)=
\begin{cases}
0 
   & \text{$x_1\leq \ell '$}\\
x_1-\ell ' 
   & \text{$x_1\in (\ell ',\ell '+1)$}\\
1 
   & \text{$x_1\geq \ell '+1.$}
\end{cases}$$
Clearly $\tilde{u}_\ell ^+\tilde{\rho}_{\ell '}$ satisfies Dirichlet boundary condition on $\partial\Omega_{\ell}^+.$ Translating this function in $\Omega_{\ell /2}$ and using definition of $\lambda_D^1(\Omega_{\ell/2})$ and $\tilde{\lambda}_M^1(\Omega_\ell ^+)$ respectively, we get
$$\lambda_D^1(\Omega_{\ell /2})\int_{\Omega_\ell ^+}|\tilde{u}_\ell ^+|^p\tilde{\rho}_{\ell '}^p
\leq\int_{\Omega_\ell ^+}|A\nabla(\tilde{u}_\ell ^+\tilde{\rho}_{\ell '})\cdot \nabla(\tilde{u}_\ell ^+\tilde{\rho}_{\ell '})|^{\frac{p}{2}}$$
and
$$\int_{\Omega_\ell ^+}|A\nabla \tilde{u}_\ell ^+\cdot \nabla \tilde{u}_\ell ^+|^{\frac{p-2}{2}}A\nabla \tilde{u}_\ell ^+\cdot \nabla(\tilde{\rho}_{\ell '}^p\tilde{u}_\ell ^+)=\tilde{\lambda}_M^1(\Omega_\ell ^+)\int_{\Omega_\ell ^+}|\tilde{u}_\ell ^+|^p\tilde{\rho}_{\ell '}^p.$$
Let $D_{\ell '}^+=\Omega_{\ell '+1}^+\setminus \Omega_{\ell '}^+$. An analogous calculation to derive \eqref{3.10} in the proof of Theorem \ref{theorem decay} gives
\begin{equation}\label{5.7}
\left(\lambda_D^1(\Omega_{\ell /2})-\tilde{\lambda}_M^1(\Omega_\ell ^+)\right)\int_{\Omega_\ell ^+}|\tilde{u}_\ell ^+|^p\tilde{\rho}_{\ell '}^p\leq p\lVert A \rVert_{\infty}^{\frac{p}{2}} C'\int_{D_{\ell '}^+}|\nabla\tilde{u}_\ell ^+|^p.
\end{equation}
By \eqref{result firoj 1}, \eqref{4.7} and the assumption \eqref{5.1}, there exists a $\tilde{\beta}'>0$ such that for $\ell >\ell_0$ we have  $\lambda_D^1(\Omega_{\ell /2})-\tilde{\lambda}_M^1(\Omega_\ell ^+)\geq \tilde{\beta}'$, which implies
\begin{equation}\label{5.8}
\tilde{\beta}'\int_{\Omega_\ell ^+}|\tilde{u}_\ell ^+|^p\tilde{\rho}_{\ell '}^p\leq p\lVert A \rVert_{\infty}^{\frac{p}{2}} C'\int_{D_{\ell '}^+}|\nabla\tilde{u}_\ell ^+|^p.
\end{equation}
Again by an analogous computation as in derivation of the \eqref{3.13} in the proof of  Theorem \ref{theorem decay}, and by the properties of $\tilde{\rho}_{\ell '}$, we obtain
$$\tilde{\beta}'\lambda^{\frac{p}{2}}\int_{\Omega_\ell ^+\setminus \Omega_{\ell '+1}^+}|\nabla\tilde{u}_\ell ^+|^p\leq p\lVert A \rVert_{\infty}^{\frac{p}{2}}\left(\lambda_M^1(\Omega_\ell )C'+\tilde{\beta}'C_p\right)\int_{D_{\ell '}^+}|\nabla\tilde{u}_\ell ^+|^p,$$
which gives
\begin{equation}\label{5.9}
\tilde{\beta}\int_{\Omega_\ell ^+\setminus \Omega_{\ell '+1}^+}|\nabla\tilde{u}_\ell ^+|^p\leq \tilde{C}\int_{\Omega_{\ell '+1}^+\setminus \Omega_{\ell '}^+}|\nabla\tilde{u}_\ell ^+|^p,
\end{equation}
where $\tilde{\beta}=\tilde{\beta}'\lambda^{\frac{p}{2}}$ and $\tilde{C}=p\lVert A \rVert_{\infty}^{\frac{p}{2}}\left(\mu^1({\omega})C'+\tilde{\beta}'C_p\right)$. From \eqref{5.9} we get
$$(\tilde{\beta}+\tilde{C})\int_{\Omega_\ell ^+\setminus \Omega_{\ell '+1}^+}|\nabla\tilde{u}_\ell ^+|^p\leq \tilde{C}\int_{\Omega_{\ell}^+\setminus \Omega_{\ell '}^+}|\nabla\tilde{u}_\ell ^+|^p,$$
which implies
\begin{equation}\label{5.10}
\int_{\Omega_\ell ^+\setminus \Omega_{\ell '+1}^+}|\nabla\tilde{u}_\ell ^+|^p\leq \tilde{\alpha}\int_{\Omega_{\ell }^+\setminus \Omega_{\ell '}^+}|\nabla\tilde{u}_\ell ^+|^p,
\end{equation}
where $\tilde{\alpha}=\frac{\tilde{C}}{\tilde{\beta}+\tilde{C}}<1$. Fix $r>1$. Applying this procedure successively for $\ell '=r-1, r-2,\cdots, r-[r]$ and then using \eqref{uniform ellip}, we obtain
\begin{multline}\label{5.11}
    \int_{\Omega_\ell ^+\setminus \Omega_{r}^+}|\nabla\tilde{u}_\ell ^+|^p\leq \tilde{\alpha}^{[r]}\int_{\Omega_{\ell }^+}|\nabla\tilde{u}_\ell ^+|^p\\
\leq \frac{\tilde{\alpha}^{[r]}}{\lambda^{\frac{p}{2}}}\int_{\Omega_\ell ^+}|A\nabla \tilde{u}_\ell ^+\cdot \nabla \tilde{u}_\ell ^+|^{\frac{p}{2}}=\frac{\tilde{\lambda}_M^1(\Omega_\ell ^+)}{\lambda^{\frac{p}{2}}}\tilde{\alpha}^{[r]}\leq \frac{\mu^1(\omega)}{\lambda^{\frac{p}{2}}}\tilde{\alpha}^{[r]}.
\end{multline}
By Poincar\'e inequality \eqref{poincare} and \eqref{5.11}, we have
$$\int_{\Omega_\ell ^+\setminus \Omega_{r}^+}|\tilde{u}_\ell ^+|^p\leq C_p^p\int_{\Omega_l^+\setminus \Omega_{r}^+}|\nabla\tilde{u}_\ell ^+|^p\leq \frac{\mu^1(\omega)C_p^p}{\lambda^{\frac{p}{2}}}\tilde{\alpha}^{[r]}=C\tilde{\alpha}^{[r]},\ \forall \ell >r,$$
which implies
\begin{equation}\label{5.12}
\int_{\Omega_{r}^+}|\tilde{u}_\ell ^+|^p\geq 1-C\tilde{\alpha}^{[r]}.
\end{equation}
Since $\tilde{u}_{\ell _k}^+\rightarrow \tilde{u}^+$ strongly in $L^p(\Omega^+_r)$, we deduce from \eqref{5.12} that
$$\int_{\Omega_{r}^+}|\tilde{u}^+|^p \geq 1-C\tilde{\alpha}^{[r]}.$$
Finally, limit $r\rightarrow\infty$ gives
\begin{equation}\label{5.13}
\int_{\Omega_{\infty}^+}|\tilde{u}^+|^p= 1.
\end{equation}
Since $\Tilde{u}^+\not\equiv0$, we can conclude that it is a minimizer realizing $\nu_{\infty}^+$ in \eqref{nu infinity}. It is a standard fact that the minimizer is unique upto multiplication by a constant and $\tilde{u}^+$ has constant sign follows from the fact that $\tilde{u}_{\ell _k}^+$ are non-negative.\smallskip

To show that $\Tilde{u}^+$ is a weak solution of \eqref{5.2}, for any fixed $\phi\in V(\Omega_{\infty}^+)$ and $t\in\mathbb{R}$ we define the function
$$J_{\phi}(t)=\frac{\int_{\Omega_{\infty}^+}|A\nabla (\tilde{u}^++t\phi)\cdot \nabla (\tilde{u}^++t\phi)|^{\frac{p}{2}}}{\int_{\Omega_{\infty}^+}|\tilde{u}
 ^++t\phi|^p}.$$
The derivative of $J_{\phi}$ with respect to $t$ is
\begin{multline*}
    J_{\phi}'(t)=\frac{\frac{p}{2}\int_{\Omega_{\infty}^+}|A\nabla (\tilde{u}^++t\phi)\cdot \nabla (\tilde{u}^++t\phi)|^{\frac{p-2}{2}}(2A\nabla\tilde{u}^+\cdot\nabla \phi+2tA\nabla \phi\cdot \nabla \phi)}{\int_{\Omega_{\infty}^+}|\tilde{u}^++t\phi|^p}\\
    -\frac{\frac{p}{2}\left(\int_{\Omega_{\infty}^+}|A\nabla (\tilde{u}^++t\phi)\cdot \nabla (\tilde{u}^++t\phi)|^{\frac{p}{2}}\right)\left(\int_{\Omega_{\infty}^+}|\tilde{u}^++t\phi|^{p-2}(2\tilde{u}^+\phi+2t\phi^2)\right)}{\left(\int_{\Omega_{\infty}^+}|\tilde{u}^++t\phi|^p\right)^2}.
\end{multline*}
Since $\Tilde{u}^+$ is a minimizer, we must have $J_{\phi}'(0)=0$. Substituting this and using \eqref{5.4} we get
\begin{equation}\label{5.14}
\int_{\Omega_{\infty}^+}|A\nabla \tilde{u}^+\cdot \nabla \tilde{u}^+|^{\frac{p-2}{2}}A\nabla \tilde{u}^+\cdot \nabla\phi=\nu_{\infty}^+
\int_{\Omega_{\infty}^+}|\tilde{u}^+|^{p-2}u\phi.
\end{equation}
This completes the proof.
\end{proof}

\begin{proof}[\textbf{Proof of Theorem \ref{theorem gap holds for mu}}]
(i) Assume that $A_{12}\cdot\nabla_{X_2}W\not\equiv0$ and \eqref{gap holds for mu} holds. For simplicity of notation, we use $\nabla W=\nabla_{X_2}W$ in this proof.  For any $\ell >0$ we define the set $T_\ell =\{x\in \omega: \mathrm{dist}(x,\partial\omega)\leq \ell \}$. Let $\beta\in (0,1)$ is fixed and $\rho_\ell $ be an approximation of the characteristic function of $\omega$, i.e. $\rho_\ell \rightarrow1$ pointwise as $\ell \rightarrow0$ and $\rho_\ell $ satisfies the following properties:
\begin{equation}\label{1.3 1}
\rho_\ell \in C_c^{\infty}(\omega),\ 0\leq\rho_\ell \leq1,\ \rho_\ell =1\ \mathrm{in} \ \omega\setminus T_\ell ,\ |\nabla\rho_\ell |\leq \frac{1}{\ell ^{\beta}}\ \mathrm{in}\ T_\ell.
\end{equation}
Define the function
$$v_\ell ^{\epsilon}(x_1,X_2)=W(X_2)-x_1\rho_\ell (X_2)G_{\epsilon}(X_2)$$
on $\Omega_\ell $, where $\{G_{\epsilon}\}\subset C_c^{\infty}(\omega)$ satisfies

$$\lim_{\epsilon\rightarrow0}G_{\epsilon}(X_2)=\frac{A_{12}(X_2)\cdot\nabla W(X_2)}{a_{11}(X_2)}\quad  \textrm{in} \ L^p(\omega).$$
We claim that
\begin{equation}\label{1.3 2}
    \inf_{\epsilon>0}\lim_{\ell \rightarrow0}\frac{\int_{\Omega_\ell ^-}|A\nabla v_\ell ^{\epsilon}\cdot \nabla v_\ell ^{\epsilon}|^{\frac{p}{2}}}{\int_{\Omega_\ell ^-}|v_\ell ^{\epsilon}|^p}\leq \int_{\omega}\left(A_{22}\nabla W\cdot\nabla W -\frac{A_{12}\cdot\nabla W}{a_{11}}\right)^{\frac{p}{2}}.
\end{equation}
 Since $$a_{11}G_{\epsilon}^2-2(A_{12}\cdot\nabla W)G_{\epsilon}+A_{22}\nabla W\cdot\nabla W\rightarrow A_{22}\nabla W\cdot\nabla W -\frac{A_{12}\cdot\nabla W}{a_{11}}$$ in $L^p(\omega)$ as $\epsilon\rightarrow0$, and L.H.S is always positive by uniform ellipticity of the matrix $A$, the function on the R.H.S of \eqref{1.3 2} is always non-negative.
The following inequality holds for any two vectors $a,$ $b$ and $q\geq1$ :
$$|b|^q\geq |a|^q+q\langle|a|^{q-2}a,b-a\rangle.$$
Using this we obtain
\begin{multline}\label{1.3 3}
 \int_{\Omega_\ell ^-}|v_\ell ^{\epsilon}|^p=\int_{\Omega_\ell ^-}|W(X_2)-x_1\rho_\ell (X_2)G_{\epsilon}(X_2)|^p\\
 \geq \int_{\Omega_\ell ^-}|W(X_2)|^p-p\int_{\Omega_\ell ^-}|W(X_2)|^{p-1}x_1\rho_\ell (X_2)G_{\epsilon}(X_2) =\ell+p\frac{\ell ^2}{2}\int_{\omega}|W|^{p-1}\rho_\ell G_{\epsilon}.
\end{multline}
Now
\begin{multline}\label{1.3 4}
 \int_{\Omega_\ell ^-}|A\nabla v_\ell^{\epsilon}\cdot \nabla v_\ell ^{\epsilon}|^{\frac{p}{2}}=\int_{\Omega_\ell ^-}|a_{11}(\partial_{x_1}v_\ell^{\epsilon})^2+2(A_{12}\cdot \nabla_{X_2}v_\ell^{\epsilon})\partial_{x_1}v_\ell^{\epsilon}+A_{22}\nabla_{X_2}v_\ell^{\epsilon}\cdot \nabla_{X_2}v_\ell^{\epsilon}|^{\frac{p}{2}}\\
 =\int_{\Omega_\ell ^-}|a_{11}\rho_\ell^2G_{\epsilon}^2-2\rho_\ell G_{\epsilon}(A_{12}\cdot \nabla W-x_1(G_{\epsilon}A_{12}\cdot \nabla \rho_\ell+\rho_\ell A_{12}\cdot \nabla G_{\epsilon}))\\
 +\left(A_{22} \nabla W-x_1(G_{\epsilon}A_{22} \nabla \rho_\ell+\rho_\ell A_{22} \nabla G_{\epsilon})\right)\cdot \left(\nabla W-x_1(G_{\epsilon} \nabla \rho_\ell+\rho_\ell \nabla G_{\epsilon})\right)|^{\frac{p}{2}}.
\end{multline}
Applying Minkowski's inequality in \eqref{1.3 4}, we get
\begin{equation}\label{1.3 5}
\int_{\Omega_\ell ^-}|A\nabla v_\ell^{\epsilon}\cdot \nabla v_\ell ^{\epsilon}|^{\frac{p}{2}}\leq ({I_1^{\epsilon}}^{\frac{2}{p}}+{I_2^{\epsilon}}^{\frac{2}{p}}+{I_3^{\epsilon}}^{\frac{2}{p}})^{\frac{p}{2}},
\end{equation}
where the integrals
\begin{equation}\label{1.3 6}
I_1^{\epsilon}=\int_{\Omega_\ell ^-}|a_{11}G_{\epsilon}^2-2(A_{12}\cdot \nabla W) G_{\epsilon}+A_{22} \nabla W\cdot \nabla W|^{\frac{p}{2}}=\ell\int_{\omega}|a_{11}G_{\epsilon}^2-2(A_{12}\cdot \nabla W) G_{\epsilon}+A_{22} \nabla W\cdot \nabla W|^{\frac{p}{2}},
\end{equation}
\begin{equation}\label{1.3 7}
I_2^{\epsilon}=\int_{\Omega_\ell^-}|a_{11}G_{\epsilon}^2(\rho_\ell^2-1)+2(A_{12}\cdot \nabla W) G_{\epsilon}(1-\rho_\ell)|^{\frac{p}{2}}=\ell K_\ell^{\epsilon},
\end{equation}
where
$$K_\ell^{\epsilon}=\int_{\omega}|a_{11}G_{\epsilon}^2(\rho_\ell^2-1)+2(A_{12}\cdot \nabla W) G_{\epsilon}(1-\rho_\ell)|^{\frac{p}{2}}$$
and
\begin{equation}\label{1.3 8}
I_3^{\epsilon}=\int_{\Omega_\ell ^-}|2x_1 H_{\ell}^{\epsilon}(X_2)+x_1^2 F_{\ell}^{\epsilon}(X_2)|^{\frac{p}{2}},
\end{equation}
where
$$H_{\ell}^{\epsilon}(X_2)=\rho_\ell G_{\epsilon}^2(A_{12}\cdot \nabla \rho_\ell)+\rho_\ell^2 G_{\epsilon}(A_{12}\cdot \nabla G_{\epsilon})-G_{\epsilon}(A_{22} \nabla W\cdot \nabla \rho_\ell)-\rho_\ell(A_{22} \nabla W\cdot \nabla G_{\epsilon}) $$
and
$$F_{\ell}^{\epsilon}(X_2)=(G_{\epsilon}A_{22} \nabla \rho_\ell+\rho_\ell A_{22} \nabla G_{\epsilon})\cdot(G_{\epsilon} \nabla \rho_\ell+\rho_\ell \nabla G_{\epsilon}).$$
Since $\rho_\ell \rightarrow1$ pointwise as $\ell \rightarrow0$, $K_\ell^{\epsilon}$ converges to $0$ as $\ell\rightarrow0$ by dominated convergence theorem. For any fix $\epsilon>0$, \eqref{1.3 1} gives
\begin{equation}\label{1.3 9}
    |H_{\ell}^{\epsilon}|\leq C_1+\frac{C_2}{\ell^{\beta}}\ \mathrm{and}\ \ |F_{\ell}^{\epsilon}|\leq C_3+\frac{C_4}{\ell^{2\beta}},
\end{equation}
where the constants $C_i>0$ ($1\leq i\leq 4$) are independent of $\ell$.
Applying Minkowski's inequality in \eqref{1.3 8}, we obtain
\begin{multline*}
I_3^{\epsilon}\leq \left(\left(\int_{\Omega_\ell ^-}|2x_1 H_{\ell}^{\epsilon}(X_2)|^{\frac{p}{2}}\right)^{\frac{2}{p}}+\left(\int_{\Omega_\ell ^-}|x_1^2 F_{\ell}^{\epsilon}(X_2)|^{\frac{p}{2}}\right)^{\frac{2}{p}}\right)^{\frac{p}{2}}\\
\leq\left(2\ell\left(\int_{\Omega_\ell ^-}|H_{\ell}^{\epsilon}(X_2)|^{\frac{p}{2}}\right)^{\frac{2}{p}}+\ell^2\left(\int_{\Omega_\ell ^-}|F_{\ell}^{\epsilon}(X_2)|^{\frac{p}{2}}\right)^{\frac{2}{p}}\right)^{\frac{p}{2}}\\
=\ell\left(2\ell\left(\int_{\omega}|H_{\ell}^{\epsilon}(X_2)|^{\frac{p}{2}}\right)^{\frac{2}{p}}+\ell^2\left(\int_{\omega}|F_{\ell}^{\epsilon}(X_2)|^{\frac{p}{2}}\right)^{\frac{2}{p}}\right)^{\frac{p}{2}}
\end{multline*}
By \eqref{1.3 9}, we get
\begin{equation}\label{1.3 10}
I_3^{\epsilon}\leq \ell\left[C_1\ell+C_2\ell^{1-\beta}+C_3\ell^2+C_4\ell^{2-2\beta}\right]^{\frac{p}{2}}=\ell C(\ell)^{\frac{p}{2}},
\end{equation}
where $C(\ell)\rightarrow0$ as $\ell\rightarrow0$. Using the estimates \eqref{1.3 6}, \eqref{1.3 7} and
\eqref{1.3 10} in \eqref{1.3 5}, we obtain
\begin{equation}\label{1.3 11}
    \int_{\Omega_\ell ^-}|A\nabla v_\ell ^{\epsilon}\cdot \nabla v_\ell ^{\epsilon}|^{\frac{p}{2}}\leq \ell \left[\left(\int_{\omega}|a_{11}G_{\epsilon}^2-2(A_{12}\cdot\nabla W)G_{\epsilon}+A_{22}\nabla W\cdot\nabla W|^{\frac{p}{2}}\right)^{\frac{2}{p}}+(K_\ell ^{\epsilon})^{\frac{2}{p}}+C(\ell ) \right]^{\frac{p}{2}}.
\end{equation}
Hence we have by \eqref{1.3 3} and \eqref{1.3 11},
\begin{equation}\label{1.3 12}
\frac{\int_{\Omega_\ell ^-}|A\nabla v_\ell ^{\epsilon}\cdot \nabla v_\ell ^{\epsilon}|^{\frac{p}{2}}}{\int_{\Omega_\ell ^-}|v_\ell ^{\epsilon}|^p}\leq\frac{\left[\left(\int_{\omega}|a_{11}G_{\epsilon}^2-2(A_{12}\cdot\nabla W)G_{\epsilon}+A_{22}\nabla W\cdot\nabla W|^{\frac{p}{2}}\right)^{\frac{2}{p}}+(K_\ell ^{\epsilon})^{\frac{2}{p}}+C(\ell ) \right]^{\frac{p}{2}}}{1+\frac{p}{2}\ell\int_{\omega}|W|^{p-1}\rho_\ell G_{\epsilon}}
\end{equation}
Sending the limits $\ell \rightarrow0$ and $\epsilon\rightarrow0$ in \eqref{1.3 12}, we prove our claim. In particular we have
$$\inf_{\epsilon>0}\lim_{\ell \rightarrow0}\frac{\int_{\Omega_\ell ^-}|A\nabla v_\ell ^{\epsilon}\cdot \nabla v_\ell ^{\epsilon}|^{\frac{p}{2}}}{\int_{\Omega_\ell ^-}|v_\ell ^{\epsilon}|^p}
<\mu^1(\omega).$$
Hence we can find some $\ell _1$ and $\epsilon_1$ for which
\begin{equation}\label{1.3 13}
-\gamma_1=\int_{\Omega_{\ell_1} ^-}|A\nabla v_{\ell _1}^{\epsilon_1}\cdot \nabla v_{\ell _1}^{\epsilon_1}|^{\frac{p}{2}}-\mu^1(\omega)\int_{\Omega_{\ell_1}^-}|v_{\ell _1}^{\epsilon_1}|^p<0.
\end{equation}
For any $\alpha>0$ we define the following function in $V(\Omega_{\infty}^+)$:
$$z_{\alpha}(x_1,X_2)=
\begin{cases}v_{\ell _1}^{\epsilon_1}(x_1-\ell _1,X_2)
   & \text{$x_1\in [0,\ell _1)$}\\
W(X_2)e^{-\alpha(x_1-\ell _1)}
   & \text{$x_1\in [\ell _1,\infty)$}.
\end{cases}$$
We have
\begin{equation}\label{1.3 14}
\int_{\Omega_{\infty}^+}|z_{\alpha}|^p=\int_{\Omega_{\ell_1}^-}|v_{\ell _1}^{\epsilon_1}|^p+\Big(\int_0^{\infty}e^{-p\alpha x_1}\Big)\int_{\omega}|W|^p=\int_{\Omega_{\ell_1}^-}|v_{\ell _1}^{\epsilon_1}|^p+\frac{1}{p\alpha},
\end{equation}
and
\begin{multline}\label{1.3 15}
\int_{\Omega_{\infty}^+}|A\nabla z_{\alpha}\cdot\nabla z_{\alpha}|^{\frac{p}{2}}=\int_{\Omega_{\ell_1}^-}|A\nabla v_{\ell _1}^{\epsilon_1}\cdot \nabla v_{\ell _1}^{\epsilon_1}|^{\frac{p}{2}}\\
+\frac{1}{p\alpha}\int_{\omega}|\alpha^2a_{11}W^2-2\alpha(A_{12}\cdot\nabla W)W+A_{22}\nabla W\cdot\nabla W|^{\frac{p}{2}}.
\end{multline}
Using \eqref{nu infinity} and then \eqref{1.3 14} and \eqref{1.3 15}, we obtain
\begin{equation}\label{1.3 16}
\nu_{\infty}^+- \mu^1(\omega)\leq \frac{\int_{\Omega_{\ell_1}^-}|A\nabla v_{\ell _1}^{\epsilon_1}\cdot \nabla v_{\ell _1}^{\epsilon_1}|^{\frac{p}{2}}\\
+\frac{1}{p\alpha}\int_{\omega}|\alpha^2a_{11}W^2-2\alpha(A_{12}\cdot\nabla W)W+A_{22}\nabla W\cdot\nabla W|^{\frac{p}{2}}}{\int_{\Omega_{\ell_1}^-}|v_{\ell _1}^{\epsilon_1}|^p+\frac{1}{p\alpha}}- \mu^1(\omega).
\end{equation}
Substituting the value $\gamma_1$ defined in \eqref{1.3 13}, we get
\begin{equation}\label{1.3 17}
\nu_{\infty}^+- \mu^1(\omega) \leq\frac{-\gamma_1+\frac{1}{p\alpha}\int_{\omega}|\alpha^2a_{11}W^2-2\alpha(A_{12}\cdot\nabla W)W+A_{22}\nabla W\cdot\nabla W|^{\frac{p}{2}}-\frac{\mu^1(\omega)}{p\alpha}}{\int_{\Omega_{\ell_1}^-}|v_{\ell _1}^{\epsilon_1}|^p+\frac{1}{p\alpha}}.
\end{equation}
Now we approximate the integral
$$I_{\alpha}=\int_{\omega}|\alpha^2a_{11}W^2-2\alpha(A_{12}\cdot\nabla W)W+A_{22}\nabla W\cdot\nabla W|^{\frac{p}{2}}.$$
Using the inequality $(a+b)^q\leq a^q+qa^{q-1}b+Ca^{q-2}b^2$ where $|b|<a$, $q\geq1$ and $C$ is some positive constant, for small $\alpha,$ we have
\begin{multline}\label{1.3 18}
I_{\alpha}\leq\int_{\omega}|A_{22}\nabla W\cdot\nabla W|^{\frac{p}{2}}+\frac{p}{2}\int_{\omega}|A_{22}\nabla W\cdot\nabla W|^{\frac{p-2}{2}}\left(\alpha^2a_{11}W^2-2\alpha(A_{12}\cdot\nabla W)W\right)\\
+C\int_{\omega}|A_{22}\nabla W\cdot\nabla W|^{\frac{p-4}{2}}\left(\alpha^2a_{11}W^2-2\alpha(A_{12}\cdot\nabla W)W\right)^2\\
=\mu^1(\omega)+\alpha^2\frac{p}{2}a_{11}I_1-\alpha pI_2+\alpha^2 CI_3,
\end{multline}

where
$$I_1=\int_{\omega}|A_{22}\nabla W\cdot\nabla W|^{\frac{p-2}{2}}W^2,$$
$$I_2=\int_{\omega}|A_{22}\nabla W\cdot\nabla W|^{\frac{p-2}{2}}(A_{12}\cdot\nabla W)W$$ and
$$I_3=\int_{\omega}|A_{22}\nabla W\cdot\nabla W|^{\frac{p-4}{2}}\left(\alpha a_{11}W^2-2(A_{12}\cdot\nabla W)W\right)^2.$$
Using \eqref{1.3 18} in \eqref{1.3 17}, we get
\begin{equation}\label{1.3 19}
\nu_{\infty}^+- \mu^1(\omega)\leq\frac{-\gamma_1+\frac{\alpha}{2} a_{11}I_1- I_2+\frac{\alpha}{p}CI_3}{\int_{\Omega_{\ell_1}^-}|v_{\ell _1}^{\epsilon_1}|^p+\frac{1}{p\alpha}}.
\end{equation}
Since $\gamma_1>0$, and $I_2\geq0$ by \eqref{gap holds for mu}, choose $\alpha$ small enough so that RHS becomes negative and this completes the proof.\smallskip

(ii) It is known that the space $V_s(\Omega_{\infty}^+)$ defined in \eqref{Vs space} is dense in $V(\Omega_{\infty}^+)$. The following subspace
$$V_s^0(\Omega_{\infty}^+)=\{u\in V_s(\Omega_{\infty}^+):\exists\delta=\delta(u)>0\ \mathrm{s.t} \ u(x)=0\ \mathrm{for} \ x \ \mathrm{with} \ \mathrm{dist}(x,\gamma_{\infty}^+)\leq \delta\}$$
is also dense in $V(\Omega_{\infty}^+)$. By regularity of the eigenfunction, we know that $W$ is continuous and positive in $\omega$ (see \cite{AnLe}). Picone's identity says that for $u\geq0$ and $v>0$, we have
\begin{equation}\label{picon}
R(u,v)=|A\nabla u\cdot\nabla u|^{\frac{p}{2}}-|A\nabla v\cdot\nabla v|^{\frac{p-2}{2}}A\nabla v\cdot\nabla\left(\frac{u^p}{v^{p-1}}\right)\geq0.
\end{equation}
Let $ u\in V_s^0(\Omega_{\infty}^+)$ be such that $u\geq0$ and let $v=W$. With these $u$ and $v$, integrating \eqref{picon} over $\Omega_{\infty}^+$ and applying integration by parts, we get
\begin{multline}\label{1.3 20}
0\leq \int_{\Omega_{\infty}^+}|A\nabla u\cdot\nabla u|^{\frac{p}{2}}-\int_{\Omega_{\infty}^+}|A\nabla W\cdot\nabla W|^{\frac{p-2}{2}}A\nabla W\cdot\nabla\left(\frac{u^p}{W^{p-1}}\right)\\
=\int_{\Omega_{\infty}^+}|A\nabla u\cdot\nabla u|^{\frac{p}{2}}+\int_{\Omega_{\infty}^+}\mathrm{div}\left(|A\nabla W\cdot\nabla W|^{\frac{p-2}{2}}A\nabla W\right)\frac{u^p}{W^{p-1}}\\
-\int_{\{0\}\times\omega}\left(|A\nabla W\cdot\nabla W|^{\frac{p-2}{2}}A\nabla W\cdot\nu\right)\frac{u^p}{W^{p-1}}\\
=\int_{\Omega_{\infty}^+}|A\nabla u\cdot\nabla u|^{\frac{p}{2}}-\mu^1(\omega)\int_{\Omega_{\infty}^+}u^p+\int_{\omega}|A_{22}\nabla W\cdot\nabla W|^{\frac{p-2}{2}}\left(A_{12}\cdot\nabla W\right)\frac{u^p(0,X_2)}{W^{p-1}(0,X_2)}.
\end{multline}
Here we have used $-\mathrm{div}(|A\nabla W\cdot\nabla W|^{\frac{p-2}{2}}A\nabla W)=\mu^1(\omega)W^{p-1}$, which is true in view of \eqref{cross}. The last integral is less than or equal to $0$ by the assumption \eqref{no gap for mu},  we have
\begin{equation}\label{1.3 21}
    0\leq \int_{\Omega_{\infty}^+}|A\nabla u\cdot\nabla u|^{\frac{p}{2}}-\mu^1(\omega)\int_{\Omega_{\infty}^+}u^p.
\end{equation}
Now let $u\in V_s^0(\Omega_{\infty}^+)$ and $u=u^+-u^-$, where $u^+(x)=$max$\{u(x),0\}$ and $u^-(x)=$min$\{0,u(x)\}$. Clearly $0\leq u^{\pm}\in V_s^0(\Omega_{\infty}^+)$ and $|u|=u^++u^-$.
Since $u^+$ and $u^-$ have disjoint supports, it is easy to check that
\begin{equation}\label{1.3 22}
\int_{\Omega_{\infty}^+}|A\nabla u\cdot\nabla u|^{\frac{p}{2}}=\int_{\Omega_{\infty}^+}|A\nabla u^+\cdot\nabla u^+|^{\frac{p}{2}}+\int_{\Omega_{\infty}^+}|A\nabla u^-\cdot\nabla u^-|^{\frac{p}{2}}
\end{equation}
and
\begin{equation}\label{1.3 23}
\int_{\Omega_{\infty}^+}|u|^p=\int_{\Omega_{\infty}^+}\left(u^+\right)^p+\int_{\Omega_{\infty}^+}\left(u^-\right)^p.
\end{equation}
Hence from \eqref{1.3 21}, \eqref{1.3 22} and \eqref{1.3 23} we obtain, for every $u\in V_s^0(\Omega_{\infty}^+)$
\begin{equation}\label{1.3 24}
\int_{\Omega_{\infty}^+}|A\nabla u\cdot\nabla u|^{\frac{p}{2}} \geq\mu^1(\omega)\int_{\Omega_{\infty}^+}u^p,\ \mathrm{i.e.}\ \ \frac{\int_{\Omega_{\infty}^+}|A\nabla u\cdot\nabla u|^{\frac{p}{2}}}{\int_{\Omega_{\infty}^+}u^p}\geq\mu^1(\omega).
\end{equation}
 By the density of the $V_s^0(\Omega_{\infty}^+)$ in $V(\Omega_{\infty}^+)$, inequality \eqref{1.3 24} holds for every $u\in V(\Omega_{\infty}^+)$  and \eqref{nu infinity} gives $\nu_{\infty}^+\geq \mu^1(\omega)$. The reverse inequality follows from Lemma \ref{lemma 4.1} and we conclude $\nu_{\infty}^+= \mu^1(\omega)$.\smallskip

Now, if $u$ be a minimizer realizing $\nu_{\infty}^+$, from \eqref{1.3 20} we must have 
$$\int_{\Omega_{\infty}^+}|A\nabla u\cdot\nabla u|^{\frac{p}{2}}=\int_{\Omega_{\infty}^+}|A\nabla W\cdot\nabla W|^{\frac{p-2}{2}}A\nabla W\cdot\nabla\left(\frac{u^p}{W^{p-1}}\right).$$
This is true only when $u=cW$ for some constant $c$. But this is a contradiction since
$W\not\in V(\Omega_{\infty}^+)$. This completes the proof.
\end{proof}
\begin{remark}\label{remark 3}
It is clear from Remark \ref{remark 1} that Theorem \ref{theorem gap holds for mu} (i) and (ii) follows for $\nu_{\infty}^-$ if we replace the inequalities \eqref{gap holds for mu} and \eqref{no gap for mu} by $\int_{\omega}|A_{22}\nabla W\cdot\nabla W|^{\frac{p-2}{2}}(A_{12}\cdot\nabla W)W\leq0$ and $A_{12}\cdot\nabla W\geq0$
respectively. In particular, if $A_{12}\cdot\nabla W\not\equiv0$ and $\int_{\omega}|A_{22}\nabla W\cdot\nabla W|^{\frac{p-2}{2}}(A_{12}\cdot\nabla W)W=0$, then $\nu_{\infty}^+< \mu^1(\omega)$ and $\nu_{\infty}^-< \mu^1(\omega)$. Hence, if $A$ satisfies the symmetry (S) defined in Definition \ref{definition S} and $A_{12}\cdot\nabla W\not\equiv0$, then it follows that $\nu_{\infty}^{\pm}< \mu^1(\omega)$. Another consequence of Theorem \ref{theorem gap holds for mu} is that if $A_{12}\cdot\nabla W\not\equiv0$, at least one of the $\nu_{\infty}^+$ and $\nu_{\infty}^-$ is strictly less than $\mu^1(\omega)$, in particular $\min \{\nu_{\infty}^+,\nu_{\infty}^-\}< \mu^1(\omega)$, which already follows from Theorem \ref{theorem itai and firoj} (ii) and Theorem \ref{theorem limit of first eigenvalue}.
\end{remark}\smallskip

In the next theorem we describe the behaviour of the eigenfunctions $\{u_\ell\}$ of mixed boundary value problem \eqref{mix} at the ends of the cylinder. Here we have shown two possible cases may happen: concentration near one of the ends of the cylinder, or near both ends. Let $\tilde{u}^{\pm}$ be the unique positive normalized minimizer in \eqref{nu infinity}, whenever it exists. For each $\ell >0$ we define the functions
$$\tilde{v}_\ell ^+(x_1,X_2)=u_\ell (x_1-\ell ,X_2)\quad on\quad \Omega_\ell ^+,$$
$$\tilde{v}_\ell ^-(x_1,X_2)=u_\ell (x_1+\ell ,X_2)\quad on\quad \Omega_\ell ^-.$$
\begin{theorem}\label{theorem 5.2}
\emph{(i)} If $\nu_{\infty}^+<\nu_{\infty}^-$, then for every $r>0$,
$$\tilde{v}_\ell ^+\longrightarrow \tilde{u}^+\ \ in\ \ W^{1,p}(\Omega_r^+)\ \ and\ \ \tilde{v}_\ell ^-\longrightarrow0\ \ in\ \ W^{1,p}(\Omega_r^-).$$
\emph{(ii)} If $A$ satisfies the symmetry (S) defined in Definition \ref{definition S} and \eqref{5.2} holds, then we have\\
$\tilde{v}_\ell ^+(x_1,X_2)=\tilde{v}_\ell ^-(-x_1,-X_2)$ and for every $r>0$,
$$\tilde{v}_\ell ^+\longrightarrow \tilde{u}^+\ \ in\ \ W^{1,p}(\Omega_r^+)\ \ and\ \ \tilde{v}_\ell ^-\longrightarrow\tilde{u}^-\ \ in\ \ W^{1,p}(\Omega_r^-).$$
\end{theorem}
\begin{proof}
(i) The convergence of $\{\tilde{v}_\ell ^-\}$ to $0$ in $W^{1,p}(\Omega_r^-)$ for any $r>0$ follows from Proposition \ref{proposition 4.1}. It remains to prove the convergence for $\{\tilde{v}_\ell ^+\}$. Let $\{\tilde{v}_{\ell_k}^+\}\subset\{\tilde{v}_\ell ^+\}$ be any sequence. Since $\{\tilde{v}_{\ell_k}^+\}$ is bounded in $W^{1,p}(\Omega_r^+)$ for any $r>0$, using a diagonal argument in $\{\tilde{v}_{\ell_k}^+\}$, we get a subsequence, we also denote it by $\{\tilde{v}_{\ell _k}^+\}$, which converges weakly in $W^{1,p}(\Omega_r^+)$ and strongly in $L^p(\Omega_r^+)$ to a function $v^+\in W^{1,p}(\Omega_{\infty}^+)$. By \eqref{decay of u} we have
\begin{equation}\label{5.40}
\int_{\Omega_r^+}|\tilde{v}_\ell ^+|^p=\int_{\Omega_\ell ^-\setminus\Omega_{\ell -r}}|u_\ell |^p=1-\int_{\Omega_{\ell -r}^-}|u_\ell |^p-\int_{\Omega_{\ell }+}|u_\ell |^p\geq1-C\alpha^{[r]}-\int_{\Omega_{\ell }+}|u_\ell |^p.
\end{equation}
Clearly the integral on the right hand side goes to $0$ as $\ell \rightarrow\infty$ by Proposition \ref{proposition 4.1}. Setting $\ell =\ell _k$ in \eqref{5.40} and taking the limit $\ell_k\rightarrow\infty$, we get
\begin{equation}\label{5.41}
    \int_{\Omega_r^+}|v^+|^p\geq1-C\alpha^{[r]}
\end{equation}
We conclude that $\int_{\Omega_{\infty}^+}|v^+|^p=1$ by taking $r$ to infinity. Now by \eqref{limit of first eigenvalue} and Fatou’s lemma, we have
\begin{multline}\label{5.42}
\nu_{\infty}^+=\lim_{\ell \rightarrow\infty}\lambda_M^1(\Omega_\ell )= \lim_{\ell \rightarrow\infty}\int_{\Omega_\ell }|A\nabla u_\ell \cdot \nabla u_\ell |^{\frac{p}{2}}\\
\geq \limsup_{k\rightarrow\infty}\int_{\Omega_r^+}|A\nabla \tilde{v}_{\ell _k}^+\cdot \nabla \tilde{v}_{\ell _k}^+|^{\frac{p}{2}}\geq\int_{\Omega_r^+}|A\nabla v^+\cdot \nabla v^+|^{\frac{p}{2}}.
\end{multline}
Hence from \eqref{5.41} and \eqref{5.42} we deduce that $\int_{\Omega_{\infty}^+}|A\nabla v^+\cdot \nabla v^+|^{\frac{p}{2}}\leq\nu_{\infty}^+$. The reverse inequality follows from \eqref{nu infinity} and we conclude that $v^+$ is a normalised minimizer realizing $\nu_{\infty}^+$. Since the minimizer is unique upto multiply by a constant, it necessarily coincides with $\tilde{u}^+$.
Now we define the following function on $(0,\infty)$,
$$f(r)=\limsup_{k\rightarrow\infty}\int_{\Omega_r^+}|A\nabla \tilde{v}_{\ell _k}^+\cdot \nabla \tilde{v}_{\ell _k}^+|^{\frac{p}{2}}-\int_{\Omega_r^+}|A\nabla v^+\cdot \nabla v^+|^{\frac{p}{2}}.$$
From \eqref{5.42} it follows  that $f$ is non-negative and non-decreasing. On the other hand, $f(r)\rightarrow0$ as $r\rightarrow\infty$. Hence $f$ must be identically zero on $(0,\infty)$. This implies $\tilde{v}_{\ell _k}^+$ converges to $\tilde{u}^+$ strongly in $W^{1,p}(\Omega_r^+)$ for any $r>0$. The convergence for the whole family $\{\tilde{v}_\ell ^+\}$ follows from the uniqueness of the all possible limits.\smallskip

(ii) If $A$ satisfies the symmetry (S), it is easy to check that $u_\ell (x_1,X_2)=u_\ell (-x_1,-X_2)$. Using this, a similar computation gives the required result for this case.
\end{proof}

\begin{section}{Convergence of eigenvalues using Krasnoselskii genus}
In this section we prove the convergence results for higher eigenvalues (Theorem \ref{theorem kth ev} and Theorem \ref{theorem second K ev}) using the Krasnoselskii genus. The following lemma shows that for $p=2$, the $k$-th Krasnoselskii eigenvalue $\beta_k(\Omega_\ell)$ of problem \eqref{mix} is same as the $k$-th eigenvalue $\lambda_M^k(\Omega_\ell)$. This fact may be known, but  for the  sake of completeness we include the proof here.
\begin{lemma}\label{lemma 6.1}
Let for any $k\in \mathbb{N}$, $\beta_k(\Omega_\ell)$ be defined as in Theorem \emph{\ref{theorem genus}} with the closed and symmetric $C^{1,1}$-manifold
$$M_\ell=\left\{u\in V(\Omega_\ell):\int_{\Omega_\ell}u^2=1 \right\}$$
of the Banach space $V(\Omega_\ell)$ and with the even functional
$$E_\ell(u)=\int_{\Omega_\ell}A\nabla u\cdot \nabla u$$
on $M_\ell$. Then $\beta_k(\Omega_\ell)=\lambda_M^k(\Omega_\ell)$, where $\lambda_M^k(\Omega_\ell)$ is the $k$-th eigenvalue of the problem \eqref{mix} for $p=2$.
\end{lemma}
\begin{proof}
Let $\left\{u_\ell^1,u_\ell^2,\cdots u_\ell^k\right\}$ be the first $k$ normalised eigenfunctions of the problem \eqref{mix}. It is well known that $u_\ell^i$'s are mutually orthogonal and
\begin{equation}\label{6.1}
\lambda_M^k(\Omega_\ell)=\underset{\underset{||u||_{L^2}=1}{u\in\mathrm{sp}\{u_\ell^1,\cdots u_\ell^k\}}}{\mathrm{max}}\,E_\ell(u)=\underset{\underset{\mathrm{dim}V=k}{V\subset V(\Omega_\ell)}}{\mathrm{min}}\,\underset{\underset{||u||_{L^2}=1}{u\in V}}{\mathrm{max}}\,E_\ell(u).
\end{equation}
Consider the subspace $V_0=\mathrm{sp}\left\{u_\ell^1,u_\ell^2,\cdots ,u_\ell^k\right\}$ of $V(\Omega_\ell)$. Define a closed symmetric subset $A_0$ of $M_\ell$ as
$$A_0=V_0\cap M_\ell=\left\{\sum_{i=1}^{k}\alpha_iu_\ell^i :\int_{\Omega_\ell}\left(\sum_{i=1}^{k}\alpha_iu_\ell^i\right)^2=1 \right\}=\left\{\sum_{i=1}^{k}\alpha_iu_\ell^i :\sum_{i=1}^{k}\alpha_i^2=1 \right\}.$$
Clearly, $V_0\cong \mathbb{R}^k$ and $A_0$ is homeomorphic to the unit sphere in $\mathbb{R}^k,$ and therefore Proposition \ref{prop genus} gives $\gamma(A_0)=k$. Now, definition of $\beta_k(\Omega_\ell)$ gives
$$\beta_k(\Omega_\ell)\leq \underset{u\in A_0}{\mathrm{sup}}\,E_\ell(u)=\underset{\underset{||u||_{L^2}=1}{u\in\mathrm{sp}\{u_\ell^1,\cdots u_\ell^k\}}}{\mathrm{sup}}\,E_\ell(u)=\lambda_M^k(\Omega_\ell).$$
For the reverse inequality, first we fix $\epsilon>0$. $\exists$ $A_0\in \mathcal{F}_k$ such that $\underset{u\in A_0}{\mathrm{sup}}E_\ell(u)<\beta_k(\Omega_\ell)+\epsilon$. Clearly by \eqref{uniform bound} and \eqref{uniform ellip}, $(u,v)_A=\int_{\Omega_\ell}A\nabla u\cdot \nabla v$ is an equivalent inner product on $V(\Omega_\ell)$. By Propotion \ref{prop genus}, let $\{w_1, w_2,\cdots w_k\}$ be a set of orthogonal (w.r.t $(\cdot,\cdot)_A$)
vectors in $A_0$.\\
 Let $V_0=\mathrm{sp}\left\{w_1, w_2,\cdots w_k\right\}$. For any $u=\alpha_1w_1+\cdots +\alpha_kw_k\in V_0$ with $||u||_{L^2(\Omega_\ell)}=1$, we have
\begin{multline}\label{6.2}
E_\ell(u)=\int_{\Omega_\ell}A\nabla u\cdot \nabla u=\int_{\Omega_\ell}A\nabla (\sum_{i=1}^k\alpha_iw_i)\cdot \nabla(\sum_{i=1}^k\alpha_iw_i)
=\sum_{i=1}^k\alpha_i^2 E_\ell(w_i)\leq \underset{1\leq i\leq k}{\mathrm{max}}E_\ell(w_i).
    \end{multline}
Hence from \eqref{6.2}
\begin{equation}
\lambda_M^k(\Omega_\ell)\leq \underset{\underset{||u||_{L^2}=1}{u\in V_0}}{\mathrm{max}}E_\ell(u)\leq \underset{1\leq i\leq k}{\mathrm{max}}\,E_\ell(w_i)\leq \underset{u\in A_0}{\mathrm{sup}}\,E_\ell(u)<\beta_k(\Omega_\ell)+\epsilon.
    \end{equation}
Since $\epsilon$ is arbitrary, this completes the proof.
\end{proof}
\begin{corollary}\label{coro 2}
For $p>2,$ we have $\beta_1(\Omega_\ell)=\lambda_M^1(\Omega_\ell)$.
\end{corollary}

\begin{proof}[\textbf{Proof of Theorem} \ref{theorem kth ev}]
First we consider the even case, i.e. the $2k$-th eigenvalue of the problem \eqref{mix}. Since $A$ satisfies the symmetry (S), $\tilde{\lambda}_M^1(\Omega_\ell ^+)=\tilde{\lambda}_M^1(\Omega_\ell ^-)$, holds $\forall\, \ell>0$ and Remark \ref{remark 1}, Lemma \ref{lemma 4.2} and Theorem \ref{theorem limit of first eigenvalue} together gives
 \begin{equation}\label{6.4}
\lim_{\ell\rightarrow\infty}\lambda_M^1(\Omega_\ell)=\lim_{\ell\rightarrow\infty}\tilde{\lambda}_M^1(\Omega_\ell ^+).
\end{equation}
Again we have $\lambda_M^1(\Omega_\ell)\leq \tilde{\lambda}_M^1(\Omega_{\ell_1} ^+)$, $\forall \ell_1\leq \ell$ by considering the test function $u\in V(\Omega_\ell)$ in \eqref{2.9} as $u(x_1,X_2)=\tilde{u}_{\ell_1}^+(x_1+\ell,X_2)$ on $\Omega_\ell^-\setminus \Omega_{\ell-\ell_1}^-$ and $0$ on $\Omega_{\ell-\ell_1}^-\cup \Omega_\ell^+$, where $\tilde{u}_{\ell_1}^+$ is the eigenfunction corresponding to $\tilde{\lambda}_M^1(\Omega_{\ell_1} ^+)$ (see Remark \ref{remark 2}). Let $\epsilon>0$. Then $\exists$ $\ell_0>0$ such that for $1\leq i\leq k$
\begin{equation}\label{6.5}
\tilde{\lambda}_M^1(\Omega_{\frac{i}{k}\ell}^+)<\lambda_M^1(\Omega_\ell)+\epsilon, \quad \forall  \ \frac{\ell}{k}>\ell_0.
    \end{equation}
    For $1\leq i\leq k$ and $\ell\geq\ell_0$, we define the following functions on $\Omega_\ell$:
\begin{equation}\label{6.6}
h_i^-(x_1,X_2)=\begin{cases}
        \tilde{u}_{\frac{i}{k}\ell}^+(x_1+\ell,X_2)
        & \text{on $\Omega_\ell^-\setminus \Omega_{(1-\frac{i}{k})\ell}^-$}\\
        0
        & \text{otherwise,}
    \end{cases}
\end{equation}
\begin{equation}\label{6.7}
h_i^+(x_1,X_2)=\begin{cases}
        \tilde{u}_{\frac{i}{k}\ell}^-(x_1-\ell,X_2)
        & \text{on $\Omega_\ell^+\setminus \Omega_{(1-\frac{i}{k})\ell}^+$}\\
        0
        & \text{otherwise,}
    \end{cases}
\end{equation}
where $\tilde{u}_{\frac{i}{k}\ell}^{\pm}$ are defined in Remark \ref{remark 2} for $p=2$. Let $u_{2i-1}=h_i^-$ and $u_{2i}=h_i^+$, where $1\leq i\leq k$ and we define the subspace $V_0=\mathrm{sp}\left\{u_1,u_2,\cdots,u_{2k}\right\}$ of $V(\Omega_\ell)$. Consider the closed symmetric subset $A_0=V_0\cap M_\ell$, where $M_\ell$ is the $C^{1,1}$-manifold defined in Lemma \ref{lemma 6.1}. Clearly $\gamma(A_0)=2k$. Now for any $u=\sum_{j=1}^{2k}\alpha_j u_j\in A_0$,
\begin{equation}\label{6.8}
    \int_{\Omega_\ell}A\nabla u\cdot \nabla u=\sum_{j=1}^{2k}\alpha_j^2\int_{\Omega_\ell}A\nabla u_j\cdot \nabla u_j + 2\sum_{1\leq i<j\leq 2k}\alpha_i\alpha_j\int_{\Omega_\ell}A\nabla u_i\cdot \nabla u_j.
\end{equation}
For $1\leq j\leq 2k$, we have
\begin{equation}\label{6.9}
    \int_{\Omega_\ell}A\nabla u_j\cdot \nabla u_j=\begin{cases}
        \tilde{\lambda}_M^1(\Omega_{\frac{j+1}{2k}\ell}^+)\int_{\Omega_\ell}u_j^2
        & \text{if $j$ is odd,}\\
        \tilde{\lambda}_M^1(\Omega_{\frac{j}{2k}\ell}^+)\int_{\Omega_\ell}u_j^2
        & \text{if $j$ is even,}
    \end{cases}
\end{equation}
and for $1\leq i<j\leq 2k$,
\begin{equation}\label{6.10}
    \int_{\Omega_\ell}A\nabla u_i\cdot \nabla u_j=\begin{cases}
        \tilde{\lambda}_M^1(\Omega_{\frac{j+1}{2k}\ell}^+)\int_{\Omega_\ell}u_iu_j
        & \text{if $j$ is odd,}\\
        \tilde{\lambda}_M^1(\Omega_{\frac{j}{2k}\ell}^+)\int_{\Omega_\ell}u_iu_j
        & \text{if $j$ is even.}
    \end{cases}
\end{equation}
Using \eqref{6.9}, \eqref{6.10} and \eqref{6.5} in \eqref{6.8} we get
\begin{multline}\label{6.11}
    \int_{\Omega_\ell}A\nabla u\cdot \nabla u\leq \left(\lambda_M^1(\Omega_\ell)+\epsilon\right)\left[\sum_{j=1}^{2k}\alpha_j^2\int_{\Omega_\ell}u_j^2 + 2\sum_{1\leq i<j\leq 2k}\alpha_i\alpha_j\int_{\Omega_\ell}u_iu_j\right]\\
    =\left(\lambda_M^1(\Omega_\ell)+\epsilon\right)\int_{\Omega_\ell}u^2=\lambda_M^1(\Omega_\ell)+\epsilon.
\end{multline}
Finally, from the definition of $\beta_{2k}(\Omega_\ell)$ we obtain
$$\lambda_M^1(\Omega_\ell)\leq \lambda_M^{2k}(\Omega_\ell)=\beta_{2k}(\Omega_\ell)\leq \underset{u\in A_0}{\mathrm{sup}}\int_{\Omega_\ell}A\nabla u\cdot \nabla u\leq \lambda_M^1(\Omega_\ell)+\epsilon.$$
Since $\epsilon$ is arbitrary, \eqref{kth ev} is true for $k$ is even. Convergence of odd eigenvalues follows  because  $$\lambda_M^{2k-2}(\Omega_\ell)\leq\lambda_M^{2k-1}(\Omega_\ell) \leq \lambda_M^{2k}(\Omega_\ell).$$
\end{proof}

\begin{proof}[\textbf{Proof of Theorem} \ref{theorem second K ev}]
Let $h_1^{\pm}$ is defined as in \eqref{6.6} and \eqref{6.7} with $k=1$. Define the closed and symmetric subset $A_0=V_0\cap M_\ell$ of genus $2$, where $V_0=\mathrm{sp}\{h_1^-,h_1^+\}$. Since $h_1^-$ and $h_1^+$ have disjoint support, for any $u=\alpha h_1^-+\beta h_1^+\in A_0$, we have
$$\int_{\Omega_\ell}|A\nabla u\cdot \nabla u|^{\frac{p}{2}}=\alpha^p\int_{\Omega_\ell}|A\nabla h_1^-\cdot \nabla h_1^-|^{\frac{p}{2}}+\beta^p\int_{\Omega_\ell}|A\nabla h_1^+\cdot \nabla h_1^+|^{\frac{p}{2}}.$$
An analogous calculation as in the proof of Theorem \ref{theorem kth ev} gives the rest of the proof.
\end{proof}

\textbf{Acknowledgement:} \ The third author would like to thank the facilities provided by IMSc. Part of this work was carried out when the third author was visiting IMSc. Research of third author is funded by Core Research grant under project number CRG/2022/007867. The third author would like to thank Prof. Itai Shafrir and Prof. Michel Chipot for several interesting discussion on the subject. In fact, the key ideas of the proof of Theorem \ref{theorem limit of first eigenvalue} and Theorem \ref{theorem gap holds for mu} are due to Prof. Shafrir. 

\textbf{Note:} This work is a part of doctoral thesis work of the second author. He acknowledges the support provided by IIT Kanpur, India and MHRD, Government of India (GATE fellowship).
\end{section}


 \end{document}